\documentclass[twocolumn,fleqn,natbib]{svjour2}
\bibpunct{[}{]}{;}{n}{}{,} 
\smartqed 
\usepackage{graphicx}
\usepackage{epsfig}
\usepackage{amsmath,amssymb,ifthen}
\usepackage{graphicx,psfrag}
\usepackage{color}

\newcommand\R{\mathbb R}
\newcommand\N{\mathbb N}

\newcommand\FF{\mathcal F}
\newcommand\HH{\mathcal H}
\newcommand\MM{\mathcal{M}}
\renewcommand\SS{\mathcal S}
\newcommand\TT{\mathcal T}
\newcommand\PP{\mathcal P}
 
\newcommand\OO{\mathcal{O}} 
\newcommand\XX{\mathcal{X}}
\newcommand\YY{\mathcal{Y}}
\newcommand\EE{\mathcal E} 
\newcommand\II{\mathcal I}

\newcommand\diam{\operatorname{diam}}

\newcommand\uu{\mathbf{u}}
\newcommand\UU{\mathbf{U}}
\newcommand\vv{\mathbf{v}}
\newcommand\VV{\mathbf{V}}
\newcommand\ww{\mathbf{w}}

\renewcommand\div{\operatorname{div}}

\newcommand\NV{\mathbf{\nu}}

\newcommand\slp{\mathfrak{V}}
\newcommand\dlp{\mathfrak{K}}
\newcommand\hyp{\mathfrak{W}}

\newcommand\A{\mathfrak{A}}

\newcommand\linfun{a}

\newcommand\chieq{\chi_{\rm eq}}
\newcommand\teq{t_{\rm eq}}

\newcommand\namecoupl[1]{%
  \ifthenelse{\equal{#1}{bmc}}{{\rm bmc}}{%
  \ifthenelse{\equal{#1}{jn}}{{\rm jn}}{%
  \ifthenelse{\equal{#1}{s}}{{\rm sym}}{%
}}}}

\newcommand\B[1]{\mathfrak{B}_{\namecoupl{#1}}}
\newcommand\BB[1]{\widetilde{\mathfrak{B}}_{\namecoupl{#1}}}
\renewcommand\b[1]{b_{\namecoupl{#1}}}
\newcommand\bb[1]{\widetilde{b}_{\namecoupl{#1}}}

\renewcommand\L{\mathfrak{L}}
\newcommand\Id{\operatorname{Id}}

\newcommand\J{\mathfrak J}

\newcommand\err{\rm err}

\newcommand\norm[2]{\lVert#1\rVert_{#2}}
\newcommand\enorm[1]{|\!|\!|#1\,\!|\!|\!|}

\newcommand\dualkl[1]{\langle #1\rangle}  
\newcommand\nnorm[1]{\lVert#1\rVert_{\HH}}
\newcommand\dual[2]{\dualkl{#1\hspace*{.5mm},#2}}

\newcommand\abs[1]{\left\lvert#1\right\rvert}

\newcommand\set[2]{\left\lbrace #1 \,\middle\vert\, #2 \right\rbrace}

\newtheorem{thm}{Theorem}

\newtheorem{lem}[thm]{Lemma}
\newtheorem{cor}[thm]{Corollary}
\newtheorem{alg}[thm]{Algorithm}
\newtheorem{ass}[thm]{Assumption}

\newtheorem{rem}[thm]{\bfseries Remark}

\newcommand\Doerfler{D\"orfler}

\newcounter{constantsnumber}
\newcommand\namec[2]{%
  \ifthenelse{\equal{#1}{reduction}}{C_{\rm red}}{%
  \ifthenelse{\equal{#1}{loc}}{C_{\rm loc}}{%
  \ifthenelse{\equal{#1}{ell}}{C_{\rm ell}}{%
  \ifthenelse{\equal{#1}{ellA}}{c_{\rm ell}}{%
  \ifthenelse{\equal{#1}{lip}}{C_{\rm Lip}}{%
  \ifthenelse{\equal{#1}{lipA}}{c_{\rm lip}}{%
  \ifthenelse{\equal{#1}{inv}}{C_{\rm inv}}{%
  \ifthenelse{\equal{#1}{rel}}{C_{\rm rel}}{%
  \ifthenelse{\equal{#1}{cea}}{C_{\mbox{\rm\scriptsize C\'ea}}}{%
  \ifthenelse{\equal{#2}{newcounter}}{\refstepcounter{constantsnumber}\label{const#1}}{}C_{\ref{const#1}}}%
}}}}}}}}}

\renewcommand\c[1]{\namec{#1}{reference}}
\journalname{Computational Mechanics}
\begin{document}

\title{Classical FEM-BEM coupling methods:\\
nonlinearities, well-posedness, and adaptivity}
\titlerunning{Classical FEM-BEM coupling methods} 

\author{Markus Aurada \and Michael Feischl \and Thomas F\"uhrer \and Michael
Karkulik \and Jens Markus Melenk \and Dirk Praetorius }

\authorrunning{Aurada, Feischl, F\"uhrer, Karkulik, Melenk, and Praetorius}

\institute{Vienna University of Technology, Institute for
Analysis and Scientific Computing, Wiedner Hauptstra\ss e 8-10,
A-1040 Vienna, Austria\\
Tel.: +43-1-58801-10154\\
Fax: +43-1-58801-10196\\
\email{Thomas.Fuehrer@tuwien.ac.at}
 }

\date{Received: date / Accepted: date}

\maketitle

\begin{abstract}
We consider a (possibly) nonlinear interface problem in 2D and 3D,
which is solved by use of various adaptive FEM-BEM coupling strategies,
namely the Johnson-N\'ed\'elec coupling, the Bielak-MacCamy coup\-ling, and Costabel's
symmetric coupling. We provide a framework to prove that the continuous as well
as the discrete Galerkin solutions of these coupling methods additionally solve
an appropriate operator equation with a Lip\-schitz continuous and strongly monotone
operator. Therefore, the coupling formulations
are well-defined, and the Galerkin solutions are quasi-optimal in the sense of
a C\'ea-type lemma. For the respective Galerkin discretizations with lowest-order polynomials,
we provide reliable
residual-based error estimators. Together with an estimator reduction
property, we prove convergence of the adaptive FEM-BEM coupling methods.
A key point for the proof of the estimator reduction are novel inverse-type
estimates for the involved boundary integral operators which are advertized.
Numerical experiments conclude the work and compare performance and effectivity
of the three adaptive coupling procedures in the presence of generic singularities.
\end{abstract}

\section{Introduction}
\label{sec:introduction}

\subsection{Model problem}\label{sec:modprob}
Let $\Omega \subseteq \R^d$ ($d=2,3$) be a bounded Lipschitz domain with
polyhedral boundary $\Gamma:= \partial \Omega$ and normal vector $\NV$.
For given data $(f,u_0,\phi_0)
\in L^2(\Omega) \times H^{1/2}(\Gamma) \times H^{-1/2}(\Gamma)$, we consider the
nonlinear interface problem
\begin{subequations}\label{eq:strongform}
\begin{align}
    -\div(\A\nabla u) &= f &&\quad\text{in }\Omega,
    \label{eq:strongform:interior}\\
    -\Delta u^{\rm ext} &= 0 &&\quad\text{in }\Omega^{\rm ext},
    \label{eq:strongform:exterior}\\
    u- u^{\rm ext} &= u_0 &&\quad\text{on }\Gamma,
    \label{eq:strongform:trace}\\
    (\A\nabla u -\nabla u^{\rm ext})\cdot\NV &= \phi_0 &&\quad\text{on }\Gamma,
    \label{eq:strongform:normal}\\
    u^{\rm ext} &= \OO(\abs{x}^{-1}) &&\quad\text{as } \abs{x}{} \rightarrow \infty.
    \label{eq:strongform:radiation}
\end{align}
\end{subequations}
As usual, these equations are understood in the weak sense, i.e. we seek for a
solution $(u,u^{\rm ext})\in H^1(\Omega)\times H_{\rm loc}^1(\Omega^{\rm ext})$,
with $H_{\rm loc}^1(\Omega^{\rm ext}) = \{v \,\vert\, v\in H^1(K)\, ,
K\subseteq\overline{\Omega^{\rm ext}} \,\text{ compact} \}$,
and $\Omega^{\rm ext} := \R^d\backslash\overline\Omega$.
It is well-known that problem \eqref{eq:strongform} admits a unique solution in
3D. In 2D, the given data have to fulfill the compatibility
condition
\begin{align}\label{eq:compatibility2d}
 \dual{f}{1}_\Omega + \dual{\phi_0}{1}_\Gamma = 0,
\end{align}
to ensure the right behaviour of the solution at infinity.
Moreover, in 2D we assume $\diam(\Omega)<1$ to ensure ellipticity of the simple-layer
potential $\slp$ defined below.
The assumptions on the strongly monotone operator $\A : \R^d \rightarrow \R^d$ will be
given in Section~\ref{sec:preliminaries:monotone}.
We denote by $H^{1/2}(\Gamma)$ the trace space of $H^1(\Omega)$ and by
$H^{-1/2}(\Gamma)$ its dual space. For simplicity, we will write $u$ for the
trace of a function $u\in H^1(\Omega)$, if the meaning is clear.

\subsection{Coupling of FEM and BEM}
Because of the presence of the unbounded exterior domain $\Omega^{\rm ext}$ in~\eqref{eq:strongform:exterior}, it is numerically
attractive to represent $u^{\rm ext}$ in terms of certain integral operators.
This leads to a contribution on the coupling boundary $\Gamma$ instead of the
exterior domain $\Omega^{\rm ext}$. For the interior domain $\Omega$ in~\eqref{eq:strongform:interior}, the (possible) nonlinearity
of $\A$ as well as the (possible) inhomogeneity $f\neq0$ favours the use of a
finite element approach. This led to the development of certain coupling procedures,
and we focus on the Johnson-N\'ed\'elec coupling~\cite{johned}, the Bielak-MacCamy
coupling~\cite{bmc}, and Costabel's symmetric coupling~\cite{costabel} in the
following. All of these approaches lead to a variational formulation
\begin{align}\label{eq:variationalform}
 b(\uu,\vv) = \L(\vv)\text{ for all }\vv\in\HH:=H^1(\Omega)\times H^{-1/2}(\Gamma)
\end{align}
with unknown solution $\uu\in\HH$, which is in a certain sense equivalent
to~\eqref{eq:strongform}.
We equip $\HH$ with the norm
\begin{align}
  \nnorm{\vv}^2 = \norm{v}{H^1(\Omega)}^2 + \norm{\psi}{H^{-1/2}(\Gamma)}^2
\end{align}
for $\vv=(v,\psi)\in\HH$.
Here, $b(\cdot,\cdot)$ is a continuous form on
$\HH\times\HH$ which is linear in the second argument, and $\L(\cdot)$ is a linear and continuous
functional on $\HH$. The original works~\cite{bmc,johned,costabel} focussed
on linear $\A$ and hence bilinear $b(\cdot,\cdot)$, and proved existence and
uniqueness of the solution $\uu\in\HH$ of~\eqref{eq:variationalform}.
In the framework of the symmetric coupling well-posedness for nonlinear $\A$ has
first been considered in the pioneering work~\cite{cs1995}.

For the Galerkin discretization, one considers a finite-dimensional and hence closed
subspace $\HH_\ell$ of $\HH$ and seeks $\UU_\ell\in\HH_\ell$ such that
\begin{align}
\label{eq:discrete:variationalform}
 b(\UU_\ell,\VV_\ell) = \L(\VV_\ell)
 \quad\text{for all }\VV_\ell\in\HH_\ell.
\end{align}
For linear $\A$, existence and uniqueness of the Galerkin solution
$\UU_\ell\in\HH_\ell$ for the symmetric coupling is already found in~\cite{costabel}. Moreover, Galerkin
solutions are quasi-optimal in the sense of the C\'ea-type lemma
\begin{align}\label{eq:cea}
     \nnorm{\uu-\UU_\ell} \leq \c{cea}
     \min_{\VV_\ell\in\HH_\ell} \nnorm{\uu-\VV_\ell},
\end{align}
where the constant $\c{cea}>0$ depends only on the geometry and on $\A$, but
is independent of the given data, the continuous solution $\uu$, and the
Galerkin solution $\UU_\ell$. The analysis of the symmetric coupling has
been generalized to nonlinear $\A$ in~\cite{cs1995}, but the proof required the
underlying mesh to be sufficiently fine, i.e.\ the \mbox{maximal} mesh-size had to be
sufficiently small. Finally, for the nonsymmetric coupling strategies
from~\cite{bmc,johned} and even linear problems, the analysis relied on the compactness of a certain integral
operator $\dlp$ involved. However, this compactness restricted the coupling boundary to
be smooth instead of piecewise polynomial.

Only very recently, Sayas~\cite{sayas09} proved that the Johnson-N\'ed\'elec
coupling is equivalent to an elliptic problem, independently of the compactness of
$\dlp$.
For linear $\A$, more precisely the Yukawa or the Laplace equation,
he thus derived that the variational formulation~\eqref{eq:variationalform} as well
as the discrete formulation~\eqref{eq:discrete:variationalform}
admit unique solutions
and that the discrete solutions are quasi-optimal in the sense of~\eqref{eq:cea}.
His analysis has been simplified by Steinbach~\cite{s2}. For quite general
linear $\A$, the latter work introduces a stabilized bilinear form
\begin{align}\label{eq:stabilization}
 \widetilde b(\widetilde\uu,\vv) := b(\widetilde\uu,\vv) + \sigma(\widetilde\uu,\vv)
\end{align}
which is proved to be elliptic provided the smallest eigenvalue of $\A$ is larger
than $1/4$. Up to some algebraic pre-/postprocessing, the solution $\uu$
of~\eqref{eq:variationalform} coincides with the solution $\widetilde\uu\in\HH$
of
\begin{align}\label{eq:stabilized}
 \widetilde b(\widetilde\uu,\vv) = \widetilde \L(\vv)
 \quad\text{for all }\vv\in\HH,
\end{align}
i.e.\ $\uu = \widetilde\uu + \uu_0$. Steinbach thus proposed to approximate the
unique solution of~\eqref{eq:stabilized} by some Galerkin solution
$\widetilde\UU_\ell\in\HH_\ell$ and to obtain an approximation of $\uu$ by
$\widetilde\UU_\ell+\uu_0$. One drawback of this method is, however, that the
computation of the stabilization $\sigma(\cdot,\cdot)$ as well as of the (constant) offset
$\uu_0$ requires the (numerical) solution of an additional integral equation
$\slp \phi_{\rm eq}=1$. Firstly, this might lead to artificial error contributions from
generic singularities of $\phi_{\rm eq}$. Secondly, the first Strang lemma comes into
play which imposes the assumption that the underlying (boundary) mesh is
sufficiently fine.

Finally, we mention the recent work~\cite{ghs09}, where for the (linear) Yukawa
equation
ellipticity of the bilinear form $b(\cdot,\cdot)$ is proved for both
the Johnson-N\'ed\'elec coupling as well as the Bielak-MacCamy coupling.

\subsection{A~posteriori error estimation}

A~posteriori error analysis aims to provide computable quantities $\varrho_\ell$
which measure the Galerkin error $\norm{\uu-\UU_\ell}\HH$ from above (reliability)
and below (efficiency). The local information provided by $\varrho_\ell$ can then
be used to refine the mesh locally, where the Galerkin error appears to be large.
For the symmetric coupling, a~posteriori error estimation
was initiated by~\cite{cs1995} for 2D and is well-established since then,
cf.\ e.g.~\cite{cfs,lmst,sm} and the references therein. To the best of our knowledge, only residual-based
error estimators provide unconditional upper bounds. On the other hand, for this type of estimators the lower bounds
still require the mesh to be globally quasi-uniform although efficiency is also observed empirically on locally refined meshes~\cite{cc}. For the Johnson-N\'ed\'elec
coupling and the 2D Laplacian, different types of a~posteriori error estimators
have recently been provided and compared in~\cite{aposterjn}.

\subsection{Contributions of current work}

Adapting the results and proofs of~\cite{ghs09,sayas09,s2}, we present a framework
which allows us to prove existence and uniqueness of the three coupling procedures
for certain nonlinear $\A$. Roughly speaking, the idea is as follows: Each
form $b(\uu,\vv)$ on $\HH$ which is linear in $\vv$, induces a nonlinear operator
$\B{}:\HH\to\HH^*$, where $\HH^*$ denotes the dual space of $\HH$.
Then, the variational
formulation~\eqref{eq:variationalform} is rewritten in operator formulation
\begin{align}\label{eq:operator}
 \B{}\uu = \L.
\end{align}
For each coupling, we introduce an appropriate stabilization $\sigma(\cdot,\cdot)$ and consider the
nonlinear operator $\BB{}$ induced by $\widetilde b(\cdot,\cdot)$
from~\eqref{eq:stabilization}. This is done in a way which ensures equivalence
\begin{align}\label{eq:equvialence}
 \B{}\uu = \L
 \quad\Longleftrightarrow\quad
 \BB{}\uu = \widetilde \L.
\end{align}
Under appropriate assumptions on $\A$, the operator $\BB{}$ is Lipschitz
continuous and strongly monotone (or: elliptic). Therefore, the continuous operator formulation $\BB{}\uu=\widetilde\L$ as well
as its Galerkin formulation admit unique solutions $\uu\in\HH$
resp.\ $\UU_\ell\in\HH_\ell$ which also
solve~\eqref{eq:variationalform} resp.~\eqref{eq:discrete:variationalform} and
satisfy the C\'ea-type estimate~\eqref{eq:cea}. For the Johnson-N\'ed\'elec coupling
and the Bielak-MacCamy coupling, our analysis requires that the ellipticity
constant $\c{ellA}>0$ of $\A$ is larger than $1/4$, which reflects the same
restriction as for the linear case in~\cite{s2}. For the symmetric coupling, we
avoid any restriction on $\c{ellA}>0$. We thus obtain the same results as
in~\cite{cs1995}, but without any restriction on the mesh-size and with a
much simpler proof. We stress that, unlike the approach of~\cite{s2}, the stabilized variant is only employed for theoretical reasons to guarantee unique solvability of the non-stabilized equations.

Finally, for lowest-order piecewise polynomials, we derive re\-si\-dual-based a~posteriori
error estimators which provide reliable upper bounds for the respective
Galerkin errors. For the Bielak-MacCamy coupling, we adapt the arguments from
our recent preprint~\cite{invest3D} to prove that the usual adaptive algorithm
drives the residual error estimator to zero.

\subsection{Outline}

We start with a preliminary Section~\ref{sec:preliminaries} which collects
the precise assumptions on $\A$, the integral operators $\slp$, $\dlp$, and
$\hyp$ involved, as well as the notation used in the remainder of the paper.
Section~\ref{sec:bmc} then considers the Bielak-MacCamy coupling. We sketch
the derivation of the coupling equations and prove existence and uniqueness of
the continuous as well as of the Galerkin formulation as outlined above.
Finally, we state and prove a residual-based a~posteriori error estimator.
In Section~\ref{sec:jn} and Section~\ref{sec:sym}, the same is done for the
Johnson-N\'ed\'elec coupling as well as for Costabel's symmetric coupling.
However, for the sake of brevity and since the proofs are very similar to that
of the Bielak-MacCamy coupling, we only sketch the details. Emphasis is laid,
however, on the fact that no restriction on the ellipticity constant $\c{ellA}>0$
of $\A$
is imposed in the case of the symmetric coupling. Section~\ref{sec:convergence}
states the usual adaptive mesh-refining algorithm. Using the concept of
estimator reduction and recent results of~\cite{invest3D}, convergence of
$\UU_\ell$ to $\uu$ is proved as $\ell\to\infty$, where $\ell$ denotes the
step counter of the adaptive loop. A final Section~\ref{sec:numerics} provides some
numerical experiments. Emphasis is laid on the comparison of the three
coupling procedures with respect to accuracy and computational time. Moreover,
we numerically investigate the restriction $\c{ellA}>1/4$ in case of the
Johnson-N\'ed\'elec and Bielak-MacCamy coupling. Finally, we see that the proposed
adaptive schemes are much superior to the usual approach, where the mesh is
only uniformly refined.

\section{Preliminaries}
\label{sec:preliminaries}

\subsection{Boundary integral operators}
\label{sec:preliminaries:bem}

Throughout, $\dlp$ denotes the double-layer potential with adjoint
$\dlp^\dagger$, $\slp$ denotes the simple-layer potential, and $\hyp$ the
hypersingular operator. With the fundamental solution of the Laplacian
\begin{align*}
  G(z) :=
  \begin{cases}
    -\tfrac1{2\pi} \log\abs{z} &\quad\text{for } z \in \R^2 \backslash \{0\}, \\
    \frac{1}{4\pi} \frac1{\abs{z}} &\quad\text{for } z \in \R^3 \backslash \{0\},
  \end{cases}
\end{align*}
these integral operators formally read as follows,
\begin{align*}
  (\slp \varphi)(x) &= \int_\Gamma G(x-y)\varphi(y)\,d\Gamma_y, \\
  (\dlp \varphi)(x) &= \int_\Gamma \partial_{\NV(y)} G(x-y)
  \varphi(y)\,d\Gamma_y, \\
  (\hyp \varphi)(x) &= -\partial_{\NV(x)} \int_\Gamma \partial_{\NV(y)}
  G(x-y)\varphi(y)
  \,d\Gamma_y,
\end{align*}
for $x\in\Gamma$ and with $\partial_{\NV(y)}$ denoting the normal derivative at $y\in
\Gamma$.
By continuous extension, we obtain bounded linear operators
\begin{align}
\begin{split}
\slp &\in L(H^{-1/2}(\Gamma) ; H^{1/2}(\Gamma)),\\
\dlp &\in L ( H^{1/2} (\Gamma) ; H^{1/2} (\Gamma) ),\\
\dlp^\dagger &\in L(H^{-1/2}(\Gamma); H^{-1/2}(\Gamma)),\\
\hyp &\in L(H^{1/2}(\Gamma); H^{-1/2}(\Gamma)).
\end{split}
\end{align}
Finally, we stress the ellipticity of the simple-layer potential
$\dual{\phi}{\slp\phi}_\Gamma \gtrsim \norm{\phi}{H^{-1/2}(\Gamma)}^2$.
Together with symmetry and continuity of $\slp$, this implies norm
equivalence
$\dual{\phi}{\slp\phi}_\Gamma \simeq \norm{\phi}{H^{-1/2}(\Gamma)}^2$.
For further properties of the integral operators, the reader is referred
to the literature, e.g.\ the monographs~\cite{hw,mclean,ss,s}.

\subsection{Strongly monotone operators}
\label{sec:preliminaries:monotone}

An operator $\BB{}:\HH\to\HH^*$ is Lipschitz continuous provided that
there is a constant $\c{lip}>0$ such that
\begin{align}\label{eq:operator:lipschitz}
    \|\BB{}\uu-\BB{}\vv\|_{\HH^*} \leq \c{lip} \nnorm{\uu-\vv}
\end{align}
holds for all $\uu,\vv\in\HH$, where $\|\cdot\|_{\HH^*}$ denotes the
usual norm on the dual space $\HH^*$. With $\dual\cdot\cdot$ the duality
brackets on $\HH^*\times\HH$, the operator $\BB{}$ is strongly monotone
provided that there is a constant $\c{ell}>0$ such that
\begin{align}\label{eq:operator:monotone}
 \dual{\BB{}\uu-\BB{}\vv}{\uu-\vv}
 \ge\c{ell}\,\nnorm{\uu-\vv}^2
\end{align}
holds for all $\uu,\vv\in\HH$. We refer to~\cite[Section~25.4]{zeidler} for
the following standard results on strongly monotone operators:
Under~\eqref{eq:operator:lipschitz}--\eqref{eq:operator:monotone}, the operator
$\BB{}$ is bijective, and the inverse $\BB{}^{-1}$ is Lipschitz continuous
with Lip\-schitz constant $1/\c{ell}$. Consequently,
for every $\widetilde\L\in\HH^*$, there is a unique $\uu\in\HH$
with
\begin{align}\label{eq:operator:continuous}
 \dual{\BB{}\uu}{\vv} = \widetilde\L(\vv)
 \quad\text{for all }\vv\in\HH,
\end{align}
and $\uu$ depends Lipschitz continuously on $\widetilde\L$. Moreover, for every closed
subspace $\HH_\ell$ of $\HH$, there is a unique $\UU_\ell\in\HH_\ell$ such
that
\begin{align}\label{eq:operator:galerkin}
 \dual{\BB{}\UU_\ell}{\VV_\ell} = \widetilde\L(\VV_\ell)
 \quad\text{for all }\VV_\ell\in\HH_\ell.
\end{align}
Finally, $\UU_\ell$ depends also Lipschitz continuously on $\widetilde\L$, and
there holds
the C\'ea-type quasi-optimality~\eqref{eq:cea}, where
$\c{cea} = \c{lip}/\c{ell}$.

\begin{rem}
Provided that the discrete spaces $\HH_\ell$ satisfy
\begin{align}\label{dp:condition}
 \HH_{\ell}\subseteq\HH_{\ell+1} 
 \text{ for all }\ell\ge0
 \quad\text{and}\quad
 \overline{\mbox{$\bigcup_{\ell=0}^\infty$}\HH_\ell} = \HH,
\end{align}
the quasi-optimality~\eqref{eq:cea} implies convergence 
$\UU_\ell\to\uu$ of the Galerkin solutions as $\ell\to\infty$.
In practice, the conditions~\eqref{dp:condition} are satisfied
if the underlying meshes are successively refined and the corresponding mesh-sizes tend to zero everywhere.
\end{rem}

To apply the framework of strongly monotone operators to the FEM-BEM
coupling formulations presented in Section~\ref{sec:bmc}, \ref{sec:jn},
and \ref{sec:sym}, we have to make some assumptions on the coefficient
function $\A : \R^d\rightarrow \R^d$.
Firstly, we consider $\A$ to be Lipschitz continuous, i.e. there
exists a constant $\c{lipA}$ such that
\begin{align}\label{eq:lip_A_p}
  |\A y-\A z| \leq \c{lipA} |y-z|
\end{align}
holds for all $y,z\in\R^d$.
Integrating the square of~\eqref{eq:lip_A_p}, one obtains
\begin{align}\label{eq:lip_A}
  \norm{\A\nabla v - \A\nabla w}{L^2(\Omega)}^2 \leq \c{lipA}^2 \norm{\nabla v - \nabla
  w}{L^2(\Omega)}^2
\end{align}
for all $v,w \in H^1(\Omega)$.
Secondly, we assume $\A$ to be strongly monotone in the following sense: There
exists a constant $\c{ellA}>0$ such that there holds
\begin{align}\label{eq:ell_A}
  \c{ellA} \norm{\nabla v - \nabla w}{L^2(\Omega)}^2 \leq \dual{\A\nabla v -
  \A\nabla w}{\nabla v - \nabla w}_\Omega
\end{align}
for all $v,w \in H^1(\Omega)$.
Here, $\dual\cdot\cdot_{\Omega}$ denotes the $L^2(\Omega)$-scalar product,
i.e.\ $\dual{v}{w}_\Omega = \int_\Omega vw\,dx$. Similarly, we shall write
$\dual\cdot\cdot_\Gamma$ for the $L^2(\Gamma)$-scalar product which is
extended by continuity to the duality bracket between $H^{1/2}(\Gamma)$
and $H^{-1/2}(\Gamma)$.

\begin{rem}
We stress that conditions~\eqref{eq:lip_A} and~\eqref{eq:ell_A} are sufficient for solvability considerations, see Sections~\ref{sec:bmc_solv}--\ref{sec:bmc_solv2},
\ref{sec:jn_solv}--\ref{sec:jn_solv2}, and \ref{sec:sym_solv}--\ref{sec:sym_solv2}. Anyhow, in our a~posteriori analysis we need that $\A$ is pointwise Lipschitz continuous~\eqref{eq:lip_A_p}.
\end{rem}

\begin{rem}
As far as existence and uniqueness of continuous solution $\uu$ and discrete solution $\UU_\ell$ from~\eqref{eq:operator:continuous}--\eqref{eq:operator:galerkin} is concerned, our analysis only requires that the induced operator $\BB{}$ is strictly monotone, i.e.\
\begin{align}\label{dp:strictlymonotone}
 \dual{\BB{}\uu\!-\!\BB{}\vv}{\uu\!-\!\vv}
 >0
 \text{ for all }\uu,\vv\in\HH
 \text{ with }\uu\neq\vv,
\end{align}
i.e.~\eqref{eq:operator:monotone} is replaced by~\eqref{dp:strictlymonotone}.
Then, the Browder-Minty theorem applies and, in particular, proves weak convergence $\UU_\ell\rightharpoonup\uu$ in $\HH$ as $\ell\to\infty$ under assumption~\eqref{dp:condition}. In this framework, however, the C\'ea-type estimate~\eqref{eq:cea} cannot hold in general and a~posteriori error estimates can hardly been derived. Therefore, we leave the details to the reader. However, we stress that~\eqref{dp:strictlymonotone} holds if the nonlinearity $\A$ satisfies
\begin{align}\label{eq:ell_A:stricly}
  0< \dual{\A\nabla v -
  \A\nabla w}{\nabla v - \nabla w}_\Omega
\end{align}
for all $v,w \in H^1(\Omega)$ with $\nabla v\neq\nabla w$ instead of~\eqref{eq:ell_A}.
\end{rem}

\subsection{Discrete spaces}\label{sec:discretization}
In Sections~\ref{sec:bmc}--\ref{sec:sym}, the model problem~\eqref{eq:strongform}
is reformulated as variational equality~\eqref{eq:variationalform} in the Hilbert space
$\HH:= H^1(\Omega) \times H^{-1/2}(\Gamma)$. For the respective
discretizations, let
$\TT_\ell$ be a regular triangulation of $\Omega$ and
let $\EE_\ell^\Gamma$ be a regular triangulation of $\Gamma$, where regularity
is understood in the sense of Ciarlet.
We approximate a function $u\in H^1(\Omega)$ by continuous, $\TT_\ell$-piecewise
affine functions on $\Omega$. For a function $\phi\in H^{-1/2}(\Gamma)$, we use
$\EE_\ell^\Gamma$-piecewise constant functions, i.e. our discrete spaces read
$\HH_\ell:=\SS^1(\TT_\ell) \times \PP^0(\EE_\ell^\Gamma) \subseteq \HH$.

Let $\EE_\ell^\Omega$ denote the set of all interior faces, i.e. for
$E\in\EE_\ell^\Omega$ there exist unique elements $T^+,T^-$ with $E = T^+\cap
T^-$. We define the patch of $E\in\EE_\ell^\Omega$ by $\omega_{\ell,E} :=
T^+\cup T^-$. Furthermore, we define the local mesh-width function $h_\ell$ by
\begin{align*}
  h_\ell(\tau) := \begin{cases} \lvert\tau\rvert^{1/d} &\quad\text{for }
  \tau\in \TT_\ell, \\
  \lvert\tau\rvert^{1/(d-1)} &\quad\text{for }\tau\in \EE_\ell^\Omega \cup
  \EE_\ell^\Gamma,
  \end{cases}
\end{align*}
where $\lvert\cdot\rvert$ denotes the volume resp. surface measure.
A triangulation $\TT_\ell$ is called $\gamma$-shape regular, if there holds
\begin{align}
  \sigma(\TT_\ell) := \max\limits_{T\in\TT_\ell} \frac{\diam(T)^d}{|T|} \leq
  \gamma.
\end{align}
Analogously we call $\EE_\ell^\Gamma$ $\gamma$-shape regular, if
\begin{align}
  \sigma(\EE_\ell^\Gamma) := \max\limits_{E\in\EE_\ell^\Gamma}
  \frac{\diam(E)^2}{|E|} \leq \gamma,
\end{align}
for $d=3$. For $d=2$, the $\gamma$-shape regularity of $\EE_\ell^\Gamma$ reads
\begin{align}
  \sigma(\EE_\ell^\Gamma) := \max\limits_{E\neq E'} \set{\frac{\abs{E'}}{\abs{E}}}
  {E'\cap E \neq \emptyset} \leq \gamma.
\end{align}
The definition of $h_\ell$ and shape regularity implies equivalence
$\diam(\tau) \simeq
h_\ell(\tau) \simeq h_\ell(\tau')$ for all $\tau,\tau' \in \TT_\ell$ resp.\ $\tau,\tau'\in\EE_\ell^\Gamma$ with $\tau\cap\tau'\neq\emptyset$, where
the hidden constants depend only on $\gamma$.

\begin{rem}(i)
  We stress that $\TT_\ell$ and $\EE_\ell^\Gamma$ are formally independent triangulations
  of $\Omega$ and $\Gamma$, respectively. For the numerical implementation, however, we
  restrict to the case that
  $\EE_\ell^\Gamma$ is the restriction $\TT_\ell|_\Gamma$ of $\TT_\ell$ on the boundary, which
  indeed is a regular triangulation of $\Gamma$. In this case, we finally remark
  that $\gamma$-shape regularity of $\TT_\ell$ also implies $\widetilde\gamma$-shape regularity
  of $\EE_\ell^\Gamma := \TT_\ell|_\Gamma$.
  \\ (ii)
  In 2D, the radiation condition~\eqref{eq:strongform:radiation} of $u^{\rm ext}$ can also be adapted to $u^{\rm ext}(x) =
  a\log|x| + \OO(1)$ for $x\rightarrow \infty$ and fixed $a\in\R$.
  In this case, the compatibility condition~\eqref{eq:compatibility2d}
 can be dropped.
  The analysis of the following sections still holds true for that case.
\end{rem}

\section{Bielak-MacCamy coupling}
\label{sec:bmc}
We can reformulate problem~\eqref{eq:strongform} with the help of the
Bielak-MacCamy FEM-BEM coupling, which first appeared in~\cite{bmc}.
This section is build up as follows:
Firstly, we give a short sketch of the derivation of the
Bielak-MacCamy coupling equations.
Then, we investigate well-posedness of
their continuous and discrete formulations.
And last, we derive an residual-based error estimator for the Bielak-MacCamy
coupling method.

\subsection{Derivation of Bielak-MacCamy coupling}
\label{sec:bmc_form}
The first Green's formula for the interior part~\eqref{eq:strongform:interior}
reads
\begin{align}\label{eq:bmc_glg1}
  \dual{\A\nabla u}{\nabla v}_\Omega - \dual{\A\nabla u \cdot \NV}{v}_\Gamma =
  \dual{f}{v}_\Omega
\end{align}
for all $v\in H^1(\Omega)$.
We plug in the jump condition~\eqref{eq:strongform:normal}
  for the normal derivative and obtain
\begin{align}\label{eq:bmc_glg1.1}
  \dual{\A\nabla u}{\nabla v}_\Omega
\!-\!
  \dual{\nabla u^{\rm ext}\cdot\NV}{v}_\Gamma = \dual{f}{v}_\Omega \!+\! \dual{\phi_0}{v}_\Gamma.
\end{align}
For the exterior solution $u^{\rm ext}$ of~\eqref{eq:strongform:exterior}, we make an indirect potential ansatz with the simple-layer potential
\begin{align}\label{eq:bmc_uext}
  u^{\rm ext} = \widetilde\slp \phi \quad\text{in }\Omega^{\rm ext},
\end{align}
where the integral operator $\widetilde\slp$ is defined as $\slp$, but is now evaluated in $\Omega^{\rm ext}$ instead of
$\Gamma$. We stress that the density $\phi\in H^{-1/2}(\Gamma)$ is unknown.
Then, we use properties of the simple-layer potential operator:
Firstly, by use of the continuity of the simple-layer potential in $\R^d$, i.e.\ $\widetilde\slp\phi = \slp\phi$ on $\Gamma$, and the
trace jump condition~\eqref{eq:strongform:trace}, we see
\begin{align}\label{eq:bmc:trace}
 -u + \slp\phi = -u_0
 \quad\text{on }\Gamma.
\end{align}
Secondly, we use the
jump condition of the exterior conormal derivative of the
  simple-layer potential
to see
\begin{align}
 \nabla u^{\rm ext}\cdot\NV
 = \partial_\NV^{\rm ext}\widetilde\slp\phi
 =-(\tfrac12-\dlp^\dagger)\phi.
\end{align}
Plugging the
last
equation into~\eqref{eq:bmc_glg1.1} and supplementing
the system with the variational formulation of~\eqref{eq:bmc:trace},
we end up with the variational formulation of the
Bielak-MacCamy coupling: Find $\uu=(u,\phi) \in \HH$ such that
\begin{align}
  \dual{\A\nabla u}{\nabla v}_\Omega + \dual{(\tfrac12-\dlp^\dagger)
  \phi}{v}_\Gamma &= \dual{f}{v}_\Omega + \dual{\phi_0}{v}_\Gamma,
  \nonumber\\\label{eq:bmc}
  -\dual{\psi}{u}_\Gamma + \dual{\psi}{\slp\phi}_\Gamma
  &= -\dual{\psi}{u_0}_\Gamma,
\end{align}
holds for all $\vv=(v,\psi) \in \HH$.

From now on, let $\XX_\ell$ be a closed subspace of $H^1(\Omega)$
and $\YY_\ell$ be a closed subspace of $H^{-1/2}(\Gamma)$.
We define $\HH_\ell := \XX_\ell \times \YY_\ell$. Note that the entire space $\HH=\HH_\ell$ is a valid choice, and hence the following analysis applies
to both, the continuous formulation~\eqref{eq:bmc} and the Galerkin
discretization. In the latter case, $\uu\in\HH$ in~\eqref{eq:bmc} is replaced
by $\UU_\ell\in\HH_\ell$, and $\vv\in\HH$ is replaced by arbitrary
$\VV_\ell\in\HH_\ell$.

\subsection{Stabilization}\label{sec:bmc_solv}

We define the linear form $\b{bmc}:\HH\times\HH \rightarrow \R$ for any $\uu =
(u,\phi), \vv = (v,\psi) \in\HH$ by
\begin{align}\label{def:b}
\begin{split}
  \b{bmc}(\uu,\vv) := \dual{\A\nabla u}{\nabla v}_\Omega &+ \dual{(\tfrac12 -
  \dlp^\dagger)\phi}{v}_\Gamma \\
  &- \dual{\psi}{u}_\Gamma + \dual{\psi}{\slp\phi}_\Gamma.
\end{split}
\end{align}
Note that $\b{bmc}(\cdot,\cdot)$ is only linear in the second argument.
Furthermore, we define linear functionals $\linfun_1$ and $\linfun_2$ on $H^1(\Omega)$
and $H^{-1/2}(\Gamma)$ by
\begin{align}
\label{def:ell_12}
  \linfun_1(v) &:= \dual{f}{v}_\Omega + \dual{\phi_0}{v}_\Gamma \\
  \linfun_2(\psi) &:= \dual{\psi}{-u_0}_\Gamma
\end{align}
for all $\vv=(v,\psi)\in\HH$. For $\linfun_2$, we also use the notation $\linfun_2(\psi) =
\dual{\psi}{\linfun_2}_\Gamma$.
With these definitions, the continuous formulation of the Bielak-MacCamy coupling is equivalently written as follows: Find $\uu\in\HH$ such that
\begin{align}
\label{eq:bmc_b}
 \b{bmc}(\uu,\vv) = \L(\vv)
 :=\linfun_1(v) + \linfun_2(\psi)
\end{align}
for all $\vv\in\HH$. Moreover, the Galerkin formulation of problem~\eqref{eq:bmc} reads: Find $\UU_\ell\in\HH_\ell$ such
that
\begin{align}
\label{eq:bmc_b_galerkin}
 \b{bmc}(\UU_\ell,\VV_\ell) = \L(\VV_\ell)
\end{align}
holds for all $\VV_\ell=(V_\ell,\Psi_\ell)\in\HH_\ell$.

Throughout the remainder of this section, we need the following assumption.
\begin{ass}\label{ass1}
There is a fixed function $\xi \in \bigcap_{\ell\in\N_0} \YY_\ell$
  with $\dual{\xi}1_\Gamma \neq 0$.
\end{ass}

\begin{rem}
  The discrete space $\YY_\ell =
  \PP^0(\EE_\ell^\Gamma)$, introduced in Section~\ref{sec:discretization},
  fulfills Assumption~\ref{ass1} with $\xi = 1$.
\end{rem}

Now, we try to show ellipticity of a linear form which is equivalent to
$\b{bmc}(\cdot,\cdot)$.
Firstly, note that we have to take care of the fact that $\b{bmc}(\cdot,\cdot)$ is not elliptic since
\begin{align}
  \b{bmc}((1,0),(1,0)) = \dual{\A\nabla 1}{\nabla 1}_\Omega = 0.
\end{align}
Therefore, we introduce a new linear form $\bb{bmc}(\cdot,\cdot)$ which is equivalent
to $\b{bmc}(\cdot,\cdot)$.

\begin{thm}\label{thm:equiv_b}
  With $\xi$ from Assumption~\ref{ass1}, the linear form
  \begin{align}\label{def:bb}
\begin{split}
    \bb{bmc}(\UU_\ell,\VV_\ell) &:= \b{bmc}(\UU_\ell,\VV_\ell)
    \\&\quad+ \dual{\xi}{\slp\Phi_\ell - U_\ell}_\Gamma
    \dual{\xi}{\slp\Psi_\ell - V_\ell}_\Gamma
\end{split}
  \end{align}
  is equivalent to the linear form $\b{bmc}(\cdot,\cdot)$ in the following sense: The pair $\UU_\ell =
  (U_\ell,\Phi_\ell) \in\HH_\ell$ solves
  problem~\eqref{eq:bmc_b_galerkin} if and only if it solves
  \begin{align}\label{eq:bmc_bb}
    \bb{bmc}(\UU_\ell,\VV_\ell) = \L(\VV_\ell) + \dual\xi{\linfun_2}_\Gamma
    \dual\xi{\slp\Psi_\ell - V_\ell}_\Gamma
  \end{align}
  for all $\VV_\ell = (V_\ell,\Psi_\ell) \in\HH_\ell$.
\end{thm}

\begin{proof}
\textbf{Step~1. }
Let $\UU_\ell=(U_\ell,\Phi_\ell)$ be a solution of~\eqref{eq:bmc_b}. Testing with $(V_\ell,\Psi_\ell) = (0,\xi)\in\HH_\ell$, we see
$\b{bmc}(\UU_\ell,(0,\xi))$\linebreak$ = \linfun_2(\xi)$. With the definition of
$\b{bmc}(\cdot,\cdot)$, we infer
\begin{align*}
  0 = \b{bmc}(\UU_\ell,(0,\xi))-\linfun_2(\xi) =
\dual\xi{\slp\Phi_\ell-U_\ell-\linfun_2}
_{\Gamma}.
\end{align*}
Hence, $\dual\xi{\slp\Phi_\ell-U_\ell-\linfun_2}_\Gamma
\dual\xi{\slp\Psi_\ell -V_\ell}_\Gamma = 0$ for all $\VV_\ell\in\HH_\ell$. Clearly, this is
equivalent to
\begin{align*}
  \dual\xi{\slp\Phi_\ell - U_\ell}_\Gamma \dual\xi{\slp\Psi_\ell -V_\ell}_\Gamma =
  \dual\xi{\linfun_2}_\Gamma \dual\xi{\slp\Psi_\ell -V_\ell}_\Gamma
\end{align*}
for all $\VV_\ell=(V_\ell,\Psi_\ell)\in\HH_\ell$. Therefore, $\UU_\ell = (U_\ell,\Phi_\ell)\in\HH_\ell$ also
solves problem~\eqref{eq:bmc_bb}.
\\
\textbf{Step~2. }
For the converse implication, let $\UU_\ell = (U_\ell,\Phi_\ell)\in\HH_\ell$ solve~\eqref{eq:bmc_bb}.
The choice of $\VV_\ell = (0,\xi)$ in~\eqref{eq:bmc_bb} gives
\begin{align*}
  \linfun_2(\xi) &+ \dual\xi{\linfun_2}_\Gamma \dual\xi{\slp\xi}_\Gamma
  =
\bb{bmc}
( (U_\ell,\Phi_\ell), (0,\xi)) \\
  &= \dual\xi{\slp\Phi_\ell-U_\ell}_\Gamma +
  \dual\xi{\slp\Phi_\ell-U_\ell}_\Gamma \dual\xi{\slp\xi}_\Gamma,
\end{align*}
which is equivalent to
\begin{align*}
  \dual\xi{\slp\Phi_\ell-U_\ell-\linfun_2}_\Gamma (1+\dual\xi{\slp\xi}_\Gamma)=0.
\end{align*}
Since $\slp$ is $H^{-1/2}(\Gamma)$-elliptic, the last equation implies
$\dual\xi{\slp\Phi_\ell-U_\ell-\linfun_2}=0$. We thus infer
\begin{align*}
  &\bb{bmc}(\UU_\ell,\VV_\ell)-\b{bmc}(\UU_\ell,\VV_\ell)
  \\&\quad
  = \dual\xi{\slp\Phi_\ell-U_\ell}_\Gamma
  \dual\xi{\slp\Psi_\ell-V_\ell}_\Gamma
  \\&\quad
  = \dual\xi{\linfun_2}_\Gamma
  \dual\xi{\slp\Psi_\ell-V_\ell}_\Gamma
\end{align*}
which, together with~\eqref{def:bb} and~\eqref{eq:bmc_bb}, proves that
$\UU_\ell = (U_\ell,\Phi_\ell)$ is also a solution of problem~\eqref{eq:bmc_b}.
\qed
\end{proof}

\subsection{Existence and uniqueness of solutions}\label{sec:bmc_solv2}

The linear forms $\b{bmc}(\cdot,\cdot)$ and $\bb{bmc}(\cdot,\cdot)$ induce operators $\B{bmc},
\BB{bmc} : \HH \rightarrow \HH^*$ by
\begin{align}\label{def:B_BB}
\begin{split}
  \dual{\B{bmc} \uu}{\vv} &:= \b{bmc}(\uu,\vv) \quad\text{for all } \uu,\vv\in\HH, \\
  \dual{\BB{bmc} \uu}{\vv} &:= \bb{bmc}(\uu,\vv) \quad\text{for all } \uu,\vv\in\HH.
\end{split}
\end{align}
The main result of this section reads as follows:

\begin{thm}\label{thm:bmc_ell}
Under Assumption~\ref{ass1} and provided that the ellipticity constant
$\c{ellA}$ of $\A$ fulfills $\c{ellA}>1/4$, the operator $\BB{bmc}$ is
strongly monotone and Lipschitz continuous.
\end{thm}

The proof requires the following lemma which is proved by means of a Rellich compactness argument.

\begin{lem}\label{lemma:equiv_norm}
For $\uu = (u,\phi)\in\HH$, let
\begin{align}
 \enorm{\uu}^2
 :=\norm{\nabla u}{L^2(\Omega)}^2 + \dual\phi{\slp\phi}_\Gamma +
 \lvert \dual\xi{\slp\phi-u}_\Gamma\rvert^2,
\end{align}
where $\xi$ is provided by Assumption~\ref{ass1}.
Then, $\enorm\cdot$ defines an equivalent norm on $\HH$.
\end{lem}

\begin{proof}
Clearly, there holds $\enorm{\uu} \lesssim \nnorm\uu$ for all
$\uu = (u,\phi) \\ \in \HH$.
To see the converse
estimate, we argue by contradiction and assume that $\nnorm{\uu_n} > n
\enorm{\uu_n}$ for certain $\uu_n = (u_n,\phi_n)$ and all $n\in \N$. We define $\vv_n = (v_n,\psi_n)$ by
\begin{align*}
  \vv_n := \frac{\uu_n}{\nnorm{\uu_n}}
\end{align*}
and obtain $\nnorm{(v_n,\psi_n)} = 1$ as well as $\enorm{(v_n, \psi_n)} < 1/n$. By
definition of $\enorm\cdot$ and ellipticity of $\slp$,
this implies $\psi_n \rightarrow 0 \in
H^{-1/2}(\Gamma)$ and $\nabla v_n \rightarrow 0 \in L^2(\Omega)$.
Moreover, by extracting a subsequence, we may assume that $(v_n, \psi_n)
\rightharpoonup (v,\psi)$ in $\HH$. Clearly, this implies $\psi = 0$ and $v_n
\rightharpoonup v \in H^1(\Omega)$, whence $v_n \rightarrow v \in L^2(\Omega)$.
Moreover, weak lower semi-continuity of $\enorm\cdot$ implies $\enorm{(v,\psi)} = 0$,
whence $\nabla v = 0$ and $|\dual{\xi}{v}_\Gamma| = 0$.
From the choice of $\xi$ and since $v$ is constant, we infer $v=0$ and thus
$v_n \rightarrow 0 \in H^1(\Omega)$. Altogether, $\vv_n=(v_n,\psi_n) \rightarrow 0
\in \HH$ contradicts $\nnorm{\vv_n} = 1$.\qed
\end{proof}

\renewcommand\teq{\phi_{\rm eq}}
The following proof of Theorem~\ref{thm:bmc_ell} is very much influenced by the investigations of~\cite{sayas09,s2}. We recall some basic
facts on the boundary integral operators, cf.\ e.g.~\cite[Chapter~6]{s}:
For given $\chi\in H^{-1/2}(\Gamma)$, let
$u_* = \widetilde\slp \chi$. Then
there holds
\begin{itemize}
  \item[$\bullet$] $\partial_\NV u_* = (\tfrac12 + \dlp^\dagger) \chi$,
  \item[$\bullet$] $\dual{\partial_\NV u_*}{v}_\Gamma = \dual{\nabla u_*}{\nabla
          v}_\Omega$
        for all $v\in H^1(\Omega)$,
 \item[$\bullet$] $\dual{\chi}{\slp\chi}_\Gamma =
        \dual{[\partial_\NV u_*]}{\slp\chi}_\Gamma =
        \norm{\nabla u_*}{L^2(\R^d)}^2$,
\end{itemize}
where the last equation is only valid for $d=2$, if $\chi \in
H_*^{-1/2}(\Gamma) = \set{\psi\in H^{-1/2}(\Gamma)}{\dual\psi1_\Gamma = 0}$.
For arbitrary $\chi \in H^{-1/2}(\Gamma)$, we therefore introduce the
splitting
\begin{align}\label{eq:phi_split}
  \chi = \chi_* + \chieq,
\end{align}
with $\chi_*\in H^{-1/2}(\Gamma)$ and $\chieq = \dual\chi1_\Gamma\, \teq$.Here,
$\teq = \slp^{-1}1/\dual{\slp^{-1}1}1_\Gamma$ is the so-called equilibrium density (or: natural density). Note that $\chi_* \in H^{-1/2}_*(\Gamma)$.
Moreover, with the commutativity relation $\dlp^\dagger\slp^{-1} = \slp^{-1}\dlp$ and the
equality $(\tfrac12 + \dlp) 1 = 0$, we infer
\begin{align}
\begin{split}
  \dual{(\tfrac12 + \dlp^\dagger)\slp^{-1}1}{v}_\Gamma &= \dual{\slp^{-1}(\tfrac12 +
  \dlp)1}{v}_\Gamma = 0
\end{split}
\end{align}
for all $v\in H^{1/2}(\Gamma)$.
Together with the splitting~\eqref{eq:phi_split}, this proves
\begin{align}
  \dual{(\tfrac12 + \dlp^\dagger)\chi}{v}_\Gamma = \dual{(\tfrac12 +
  \dlp^\dagger)\chi_*}{v}_\Gamma.
\end{align}
Finally, there holds
\begin{align*}
  \dual{\chi_*}{\slp\chieq}_\Gamma = \dual{\chi_*}1_\Gamma \dual{\chi}1_\Gamma /
  \dual{\slp^{-1}1}1_\Gamma = 0
\end{align*}
and therefore
\begin{align}
\begin{split}\label{eq:splitting:bmc3}
  &\dual{\chi}{\slp\chi}_\Gamma
  = \dual{\chi_*}{\slp\chi_*}_\Gamma + \dual{\chieq}{\slp\chieq}_\Gamma.
\end{split}
\end{align}

\begin{proof}[of Theorem~\ref{thm:bmc_ell}]
Lipschitz continuity of $\BB{bmc}$ simply follows from the Lipschitz continuity of
$\A$ and the continuity of the boundary integral operators.

It thus only remains to show ellipticity of $\BB{bmc}$.
Let $\uu=(u,\phi),\vv=(v,\psi)\in\HH$. Then
\begin{align}\label{eq:bmc_ell_1}
  &\dual{\BB{bmc}\uu-\BB{bmc}\vv}{\uu-\vv}
  \nonumber\\&\quad
  = \dual{\A\nabla u - \A\nabla v}{\nabla u -
  \nabla v}_\Omega
  \nonumber\\&\qquad
  + \dual{(\tfrac12-\dlp^\dagger)(\phi-\psi)}{u-v}_\Gamma
  - \dual{\phi-\psi}{u-v}_\Gamma
  \nonumber\\&\qquad
  + \dual{\phi-\psi}{\slp(\phi-\psi)}_\Gamma
  + \lvert \dual{\xi}{\slp(\phi-\psi)-(u-v)}_\Gamma\rvert^2
  \nonumber\\&\quad
  =: I_1 + I_2 + I_3 + I_4 + I_5.
\end{align}
Below, we show
\begin{align*}
 I_1 \!+\! I_2 \!+\! I_3 \!+\! I_4
 \gtrsim \norm{\nabla u\!-\!\nabla v}{L^2(\Omega)}^2
 + \dual{\phi\!-\!\psi}{\slp(\phi\!-\!\psi)}_\Gamma.
\end{align*}
With Lemma~\ref{lemma:equiv_norm} and the definition of $I_5$, this implies
\begin{align*}
 \dual{\BB{bmc}\uu-\BB{bmc}\vv}{\uu-\vv}
 \gtrsim \enorm{\uu-\vv}^2
\end{align*}
and thus concludes the proof.
\\\textbf{Step~1. }%
To abbreviate the notation, we write $\ww = (w,\chi) \\ = \uu-\vv$. The term $I_1$ is estimated by strong monotonicity~\eqref{eq:ell_A} of $\A$,
\begin{align}\label{eq:bmc_ell_2}
  \dual{\A\nabla u - \A\nabla v}{\nabla w}_\Omega \geq \c{ellA}
  \norm{\nabla w}{L^2(\Omega)}^2.
\end{align}
\textbf{Step~2. }%
With the splitting~\eqref{eq:phi_split} of $\chi$ and $u_* =
\widetilde\slp\chi_*$,
the terms $I_2+I_3$ can be estimated by
\begin{align*}
  I_2 + I_3 = &-\dual{(\tfrac12+\dlp^\dagger)\chi}{w}_\Gamma
  =-\dual{(\tfrac12+\dlp^\dagger)\chi_*}{w}_\Gamma \\
  & = -\dual{\partial_\NV u_*}{w}_\Gamma
  = -\dual{\nabla u_*}{\nabla w}_\Omega \\
  &\geq -\norm{\nabla u_*}{L^2(\Omega)} \norm{\nabla w}{L^2(\Omega)} \\
  &\geq -\norm{\nabla u_*}{L^2(\R^d)} \norm{\nabla
  w}{L^2(\Omega)}.
\end{align*}
\textbf{Step~3. }%
We recall Young's inequality: For arbitrary $a,b\in\R,\delta>0$ there holds
$ab\leq \tfrac\delta{2} a^2 + \tfrac{\delta^{-1}}{2} b^2$. We infer
\begin{align}\label{eq:bmc_ell_3}
\begin{split}
  &-\norm{\nabla u_*}{L^2(\R^d)}\norm{\nabla w}{L^2(\Omega)} \\
  &\quad\geq -\tfrac\delta{2} \norm{\nabla w}{L^2(\Omega)}^2 -
  \tfrac{\delta^{-1}}{2} \norm{\nabla u_*}{L^2(\R^d)}^2.
\end{split}
\end{align}
We combine the second term with $I_4$ and see
\begin{align}\label{eq:bmc_ell_35}
  &-\tfrac{\delta^{-1}}{2} \norm{\nabla u_*}{L^2(\R^d)}^2 +
  \dual{\chi}{\slp\chi}_\Gamma \nonumber\\
  &\quad
  = -\tfrac{\delta^{-1}}{2}\dual{\chi_*}{\slp\chi_*}_\Gamma +
  \dual{\chi_*}{\slp\chi_*}_\Gamma +
  \dual{\chieq}{\slp\chieq}_\Gamma \nonumber\\
  &\quad
  = (1-\tfrac{\delta^{-1}}2) \dual{\chi_*}{\slp\chi_*}_\Gamma +
  \dual{\chieq}{\slp\chieq}_\Gamma \nonumber\\
  &\quad
  \geq (1-\tfrac{\delta^{-1}}2)\,\dual{\chi}{\slp\chi}_\Gamma,
\end{align}
\textbf{Step~4. }%
We combine~\eqref{eq:bmc_ell_1}--\eqref{eq:bmc_ell_35} and obtain
\begin{align*}
\begin{split}
  \dual{\BB{bmc}\uu-\BB{bmc}\vv}{\ww} &
  \geq (\c{ellA}-\tfrac\delta{2})\, \norm{\nabla w}{L^2(\Omega)}^2
  \\&
  + (1-\tfrac{\delta^{-1}}2)\,
  \dual{\chi}{\slp\chi}_\Gamma
  \\&
  + \lvert \dual{\xi}{\slp\chi-w}_\Gamma\rvert^2.
\end{split}
\end{align*}
We have assumed that $\c{ellA}>1/4$. Hence, there exists some $\delta > 0$ with
$1/4<\delta/2<\c{ellA}$. Furthermore, such a $\delta$ implies
$\c{ellA}-\tfrac\delta{2}>0$ as well as $1-\tfrac{\delta^{-1}}{2}>0$.
We define
$\c{ell} := \min\{\c{ellA}-\tfrac\delta{2},1-\tfrac{\delta^{-1}}2\}$ and
end up with
\begin{align*}
  \dual{\BB{bmc}\uu-\BB{bmc}\vv}{\ww} &\geq \c{ell} \,\enorm{\ww}^2.
\end{align*}
With Lemma~\ref{lemma:equiv_norm}, this proves ellipticity of $\BB{bmc}$.\qed
\end{proof}

\begin{rem}
  (i)
  In the case $d=3$, there holds
  \begin{align*}
    \norm{\nabla u_*}{L^2(\R^d)}^2 = \dual{\chi}{\slp\chi}_\Gamma
  \end{align*}
  for all $\chi\in H^{-1/2}(\Gamma)$ and $u_* = \widetilde\slp \chi$. Then, the
  proof of Theorem~\ref{thm:bmc_ell} simplifies, because the
  splitting~\eqref{eq:phi_split} is not needed and one may simply choose $\chi=\chi_*$.
  \\(ii)
  The assumption $\c{ellA}>1/4$ is sufficient, but may not be necessary.
  Numerical experiments for a linear operator $\A$ have shown
  that the bound $1/4$ is not sharp, i.e. solvability
  seems to be given also for $0<\c{ellA}\le1/4$.
\end{rem}

Finally, we may apply the standard results from the theory on strongly monotone
operators, see Section~\ref{sec:preliminaries:monotone}, to prove in
conjunction with Theorem~\ref{thm:equiv_b} the following corollary.

\begin{cor}\label{cor:cea}
Under the assumptions of Theorem~\ref{thm:bmc_ell}, the Bielak-MacCamy
coupling~\eqref{eq:bmc_b} admits a unique solution $\uu\in\HH$. Moreover,
Galerkin approximations $\UU_\ell\in\HH_\ell$ of~\eqref{eq:bmc_b_galerkin} are
quasi-optimal in the sense of~\eqref{eq:cea}.
\end{cor}

\begin{proof}
We define $\BB{} := \BB{bmc}$, and let $\widetilde\L$ be the right-hand side
of~\eqref{eq:bmc_bb}. According to the main theorem on strongly monotone
operators, the operator equation~\eqref{eq:operator:continuous} and its Galerkin
discretization~\eqref{eq:operator:galerkin} admit unique solutions $\uu\in\HH$
and $\UU_\ell\in\HH_\ell$. Moreover, these satisfy the
quasi-optimality~\eqref{eq:cea}. Finally, Theorem~\ref{thm:equiv_b} proves that $\uu\in\HH$ is the unique solution of~\eqref{eq:bmc_b}, and $\UU_\ell\in\HH_\ell$ is the unique solution of~\eqref{eq:bmc_b_galerkin}.
\qed
\end{proof}

\subsection{Residual-based error estimator}\label{sec:bmc_errest}
Our aim is to derive a reliable residual-based error estimator for the
Bielak-MacCamy coupling in the same manner as in e.g~\cite{cs1995}
or~\cite{aposterjn}.

Let $[\A\nabla U_\ell \cdot \NV]\vert_E$ denote the jump of $\A\nabla U_\ell
\cdot \NV$ over the interior face $E\in\EE_\ell^\Omega$.
We assume additional regularity $\phi_0 \in L^2(\Gamma)$ and $u_0
\in H^1(\Gamma)$ from now on.

\begin{thm}
  \label{thm:bmc_est}
  Suppose that $\uu \in\HH$ is the unique solution of the Bielak-MacCamy
  coupling~\eqref{eq:bmc}, and $\UU_\ell\in\HH_\ell:=\SS^1(\TT_\ell)\times\PP^0(\EE_\ell^\Gamma)$ is its Galerkin
  approximation. Then, there holds
  \begin{align*}
    \c{rel}^{-2} \nnorm{\uu-\UU_\ell}^2 \leq \rho_\ell^2 := \sum_{T\in\TT_\ell} \rho_\ell(T)^2 +
    \sum_{E\in\EE_\ell^\Omega \cup \EE_\ell^\Gamma} \rho_\ell(E)^2.
  \end{align*}
  The volume contributions read
  \begin{align}\label{eq:bmc_errest_vol1}
    \rho_\ell(T)^2 &=h_T^2\norm{f}{L^2(T)}^2 \quad\text{for } T \in\TT_\ell, \\
    \label{eq:bmc_errest_vol2}
    \rho_\ell(E)^2 &=h_E\norm{[\A\nabla U_\ell\cdot \NV]}{L^2(E)}^2
    \quad\text{for } E\in\EE_\ell^\Omega,
  \end{align}
  whereas the boundary contributions read
  \begin{align*}
    \rho_\ell(E)^2 &= h_E\norm{\phi_0 + (\dlp^\dagger-\tfrac12) \Phi_\ell -
        \A\nabla U_\ell\cdot\NV}{L^2(E)}^2 \\
	&\hspace{1cm}+h_E\norm{\nabla_\Gamma(U_\ell - u_0 -\slp\Phi_\ell)}{L^2(E)}^2	
  \end{align*}
  for $E\in\EE_\ell^\Gamma$.
  The constant $\c{rel} >0$ depends only on $\Omega$, $\Gamma$
  and the $\gamma$-shape regularity of $\TT_\ell$ and $\EE_\ell^\Gamma$.
  The symbol $\nabla_\Gamma (\cdot)$ denotes the surface gradient (resp.
  arclength derivative for $d=2$).
\end{thm}

\begin{proof}
  Recall the
definitions~\mbox{\eqref{def:ell_12}--\eqref{eq:bmc_b}}
and~\eqref{def:B_BB} of
  $\B{bmc}$, $\L$, and $\BB{bmc}$.
  Problem~\eqref{eq:bmc} for the exact solution and its Galerkin approximation are equivalently written as
  \begin{align*}
    \dual{\B{bmc}\uu}{\vv} = \L(\vv) \quad\text{for all } \vv \in \HH,
  \end{align*}
  and
    \begin{align*}
    \dual{\B{bmc}\UU_\ell}{\VV_\ell} = \L(\VV_\ell) \quad\text{for all } \VV_\ell \in \HH_\ell.
  \end{align*}
We stress that $\dual{\BB{bmc}\uu-\BB{bmc}\UU_\ell}{\uu-\UU_\ell} -
\dual{\B{bmc}\uu-\B{bmc}\UU_\ell}{\uu-\UU_\ell} =
\lvert\dual{\xi}{\slp(\phi-\Phi_\ell)-(u-U_\ell)}_\Gamma \rvert^2 = 0$,
which follows from the second equation of the Bielak-MacCamy coupling, cf.~\eqref{eq:bmc}. With ellipticity of $\BB{bmc}$, we obtain
\begin{align*}
  \nnorm{\uu-\UU_\ell}^2 &\lesssim \dual{\BB{bmc}\uu-\BB{bmc}\UU_\ell}{\uu-\UU_\ell} \\
  &= \dual{\B{bmc}\uu-\B{bmc}\UU_\ell}{\uu-\UU_\ell} \\
  &\leq \norm{\B{bmc}\uu-\B{bmc}\UU_\ell}{\HH^*}\nnorm{\uu-\UU_\ell}.
\end{align*}
This leads us to
\begin{align*}
  \nnorm{\uu-\UU_\ell} & \lesssim \norm{\B{bmc}\uu-\B{bmc}\UU_\ell}{\HH^*}
  \\&= \sup_{\vv\in\HH\backslash \{0\}}
  \frac{\abs{\L(\vv) - \dual{\B{bmc}\UU_\ell}{\vv}}}{\nnorm{\vv}} \\
  &= \sup_{\vv\in\HH\backslash \{0\}}\frac{\abs{\L(\vv-\VV_\ell) -
  \dual{\B{bmc}\UU_\ell}{\vv-\VV_\ell}}}{\nnorm{\vv}}
\end{align*}
for arbitrary $\VV_\ell \in \HH_\ell$.
We choose $\VV_\ell = (\J_\ell v, 0) \in \HH_\ell$, where $\J_\ell: H^1(\Omega)
\rightarrow \SS^1(\TT_\ell)$ denotes a Cl\'ement-type quasi-interpolation
operator, which satisfies a local first-order approximation property
\begin{align}\label{eq:intop_approx}
  \norm{w-\J_\ell w}{L^2(T)} \lesssim \diam(T) \norm{\nabla w}{L^2(\omega_T)}
\end{align}
and local $H^1$-stability
\begin{align}\label{eq:intop_h1stab}
  \norm{\nabla(w-\J_\ell w)}{L^2(T)} \lesssim \norm{\nabla w}{L^2(\omega_T)}
\end{align}
for all $w \in H^1(\Omega)$ and $T\in\TT_\ell$.
Here, $\omega_T = \bigcup \{ T' \in\TT_\ell \, \vert\, T'\cap T \neq \emptyset\}$
denotes the patch of an element $T\in\TT_\ell$.
An example for such an operator $\J_\ell$ is the Cl\'ement operator~\cite{ao00} or
the Scott-Zhang projection~\cite{scottzhang}.
Note that the constants in the
estimates~\eqref{eq:intop_approx}--\eqref{eq:intop_h1stab} only depend on $\Omega$ and the
$\gamma$-shape regularity of $\TT_\ell$.
An immediate consequence of these properties and the
trace inequality is that
\begin{align}\label{eq:intop_trace}
  \norm{w-\J_\ell w}{L^2(E)} \lesssim \diam(E)^{1/2} \norm{w}{H^1(\omega_T)},
\end{align}
where $E$ is a face of the element $T \in \TT_\ell$ with $E\subseteq T$.
Note that
\begin{align}
  \begin{split}\label{eq:errest_1}
    &\L(\vv-\VV_\ell) - \dual{\B{bmc}\UU_\ell}{\vv-\VV_\ell} \\
    &\quad= \dual{f}{v-\J_\ell v}_\Omega
    -\dual{\A\nabla U_\ell}{\nabla(v-\J_\ell v)}_\Omega \\
    &\qquad +\dual{\phi_0 - (\tfrac12-\dlp^\dagger) \Phi_\ell}{v - \J_\ell v}_\Gamma
    \\&\qquad+ \dual{\psi}{U_\ell-u_0-\slp\Phi_\ell}_\Gamma.
  \end{split}
  \end{align}
  The first term on the right-hand side is estimated by use of Cauchy
  inequalities and~\eqref{eq:intop_approx}. This gives
    \begin{align*}
      &\dual{f}{v-\J_\ell v}_\Omega = \sum_{T\in\TT_\ell} \dual{f}{v-\J_\ell
      v}_T \\
      &\quad\leq \sum_{T\in\TT_\ell} \norm{f}{L^2(T)}\norm{v-\J_\ell v}{L^2(T)} \\
      &\quad\lesssim \sum_{T\in\TT_\ell} \diam(T)\norm{f}{L^2(T)}
      \norm{\nabla v}{L^2(\omega_T)} \\
      &\quad\lesssim \Big(\sum_{T\in\TT_\ell} \norm{h_\ell f}{L^2(T)}^2 \Big)^{1/2}
      \Big(\sum_{T\in\TT_\ell} \norm{\nabla v}{L^2(\omega_T)}^2 \Big)^{1/2} \\
      &\quad\lesssim \norm{h_\ell f}{L^2(\Omega)} \norm{\nabla v}{L^2(\Omega)}.
    \end{align*}
Piecewise integration by parts of the second term on the right-hand side of
\eqref{eq:errest_1} yields
  \begin{align*}
    &\dual{\A\nabla U_\ell}{\nabla(v-\J_\ell v)}_\Omega = \sum_{T\in\TT_\ell}
    \dual{\A\nabla U_\ell}{\nabla(v-\J_\ell v)}_T \\
    &= \sum_{T\in\TT_\ell} \dual{-\div\A\nabla U_\ell}{v-\J_\ell v}_T 
    +\dual{\A\nabla U_\ell \cdot \NV}{v-\J_\ell v}_{\partial T} \\
    &= \dual{\A\nabla U_\ell \cdot \NV}{v-\J_\ell v}_\Gamma 
    + \sum_{E\in\EE_\ell^\Omega}
    \dual{[\A\nabla U_\ell\NV]}{v-\J_\ell v}_E,
  \end{align*}
  since $\div\A\nabla U_\ell$ vanishes elementwise because of
  $\A\nabla U_\ell \in \PP^0(\TT_\ell)$.
  The second and third term on the right-hand side of~\eqref{eq:errest_1}
  can thus be estimated by
  \begin{align*}
    &\dual{\phi_0 - (\tfrac12-\dlp^\dagger) \Phi_\ell-\A\nabla U_\ell\cdot\NV}{v
    - \J_\ell v}_\Gamma \\
    & \hspace{1cm}- \sum_{E\in\EE_\ell^\Omega}
    \dual{[\A\nabla U_\ell\cdot\NV]}{v-\J_\ell v}_E \\
    &\leq \sum_{E\in\EE_\ell^\Gamma} \norm{\phi_0 -
    (\tfrac12-\dlp^\dagger) \Phi_\ell - \A\nabla U_\ell\cdot\NV}{L^2(E)}
    \norm{v-\J_\ell v}{L^2(E)} \\
    &+ \sum_{E\in\EE_\ell^\Omega} \norm{[\A\nabla U_\ell\cdot\NV]}
    {L^2(E)}\norm{v-\J_\ell v}{L^2(E)} \\
    &=: J_1 + J_2
  \end{align*}
  For each boundary face $E$ we fix the unique element $T_E$ with $E \subseteq
  T_E$ and infer by use of~\eqref{eq:intop_trace}
  \begin{align*}
     J_1
     &\lesssim \sum_{E\in\EE_\ell^\Gamma} \diam(E)^{1/2} \\
     &\hspace{1cm} \norm{\phi_0 -
     (\tfrac12-\dlp^\dagger) \Phi_\ell - \A\nabla U_\ell\cdot\NV}{L^2(E)}
     \norm{v}{H^1(\omega_{T_E})} \\
      &\lesssim \norm{h_\ell^{1/2}(\phi_0-(\tfrac12-\dlp^\dagger) \Phi_\ell
      -\A\nabla U_\ell\cdot\NV)}{L^2(\Gamma)} \norm{v}{H^1(\Omega)}.
  \end{align*}
  For an interior face $E$ we fix some element $T_E$ with $E \subseteq T_E$ and
  estimate the term $J_2$ analogously by
  \begin{align*}
    J_2 \lesssim \Big(
    \sum_{E\in\EE_\ell^\Omega} \norm{h_\ell^{1/2} [\A\nabla
    U_\ell \cdot\NV]}{L^2(E)}^2\Big)^{1/2} \norm{v}{H^1(\Omega)}.
  \end{align*}
 To estimate the fourth term in~\eqref{eq:errest_1} we observe
  \begin{align*}
    \dual{1}{U_\ell-u_0-\slp\Phi_\ell}_E = 0
  \end{align*}
  for any $E\in\EE_\ell^\Gamma$, which follows from the
  second equation of~\eqref{eq:bmc} tested with the characteristic function of
  $E$, which belongs to $\PP^0(\EE_\ell^\Gamma)$.
  Now,~\cite[Corollary~4.2]{cms} can be applied and proves
  \begin{align*}
    &\norm{U_\ell-u_0-\slp\Phi_\ell}{H^{1/2}(\Gamma)}^2 \\
    &\hspace{0.5cm}\leq \c{loc}
    \sum_{E\in\EE_\ell^\Gamma} \diam(E)\, \norm{\nabla_\Gamma
    (U_\ell-u_0-\slp\Phi_\ell)}{L^2(E)}^2,
  \end{align*}
  where $\c{loc} > 0$ depends only on $\Gamma$ and the $\gamma$-shape
  regularity of $\EE_\ell^\Gamma$.
  This leads us to
  \begin{align*}
    &\dual{\psi}{U_\ell-u_0-\slp\Phi_\ell}_\Gamma \\
    &\hspace{0.5cm} \leq \norm{\psi}{H^{-1/2}(\Gamma)}
    \norm{U_\ell-u_0-\slp\Phi_\ell}{H^{1/2}(\Gamma)} \\
    &\hspace{0.5cm} \lesssim \norm{\psi}{H^{-1/2}(\Gamma)}
    \norm{h_\ell^{1/2}\nabla_\Gamma(U_\ell -u_0 -
    \slp\Phi_\ell)}{L^2(\Gamma)}.
  \end{align*}
  In 2D, the same estimate can also be obtained by use of the continuity of
  $U_\ell- u_0-\slp\Phi_\ell\in H^1(\Gamma)$, cf.~\cite{cs1995}.
  Altogether, we have
  \begin{align*}
    &\abs{\L(\vv-\VV_\ell)-\dual{\B{bmc}\UU_\ell}{\vv-\VV_\ell}}/
\nnorm{\vv}\\
    &\hspace{0.5cm}\lesssim \norm{h_\ell f}{L^2(\Omega)} +
    \norm{h_\ell^{1/2} [\A\nabla U_\ell
    \cdot\NV]}{L^2(\bigcup\EE_\ell^\Omega)}
    \\ & \hspace{1cm}+ \norm{h_\ell^{1/2}(\phi_0-(\tfrac12-\dlp^\dagger)
    \Phi_\ell - \A\nabla U_\ell\cdot\NV)}{L^2(\Gamma)} \\
    &\hspace{1cm}+\norm{h_\ell^{1/2}\nabla_\Gamma(U_\ell
    - u_0 - \slp\Phi_\ell)}{L^2(\Gamma)} ,
  \end{align*}
  which concludes the proof.\qed
\end{proof}

\section{Johnson-N\'ed\'elec coupling}
\label{sec:jn}
In this section, we present the Johnson-N\'ed\'elec coupling, which first
appeared in~\cite{johned}. As in the previous section, we state the continuous and discrete
formulation of this method and discuss existence and uniqueness of the corresponding solutions.
Finally, we provide a reliable resi\-du\-al-based error
estimator.

\subsection{Derivation of Johnson-N\'ed\'elec coupling}
Unlike the Bielak-MacCamy coupling, we represent the exterior solution
$u^{\rm ext}$ by use of the third Green's identity in the exterior domain $\Omega^{\rm ext}$,
\begin{align}\label{eq:jn_rep_form}
  u^{\rm ext} = \widetilde\dlp u^{\rm ext} - \widetilde\slp\phi,
\end{align}
where we define $\phi = \partial_\NV u^{\rm ext}$.
As above $\widetilde\slp$ and $\widetilde\dlp$ are defined as $\slp$ and $\dlp$, but are now evaluated in $\Omega^{\rm ext}$ instead of $\Gamma$.
Taking the trace in~\eqref{eq:jn_rep_form}, we see
\begin{align}
u^{\rm ext} = (\dlp+\tfrac12)u^{\rm ext} - \slp\phi.
\end{align}
Using the trace jump condition~\eqref{eq:strongform:trace}, we obtain
\begin{align}
 (\tfrac12-\dlp)u + \slp \phi = (\tfrac12-\dlp)u_0.
\end{align}
Together with~\eqref{eq:bmc_glg1.1}, the variational formulation of the latter equation provides the Johnson-N\'ed\'elec coupling:
Find $\uu=(u,\phi) \in \HH$ such that
  \begin{align}
    \begin{split}\label{eq:jn}
      \dual{\A\nabla u}{\nabla v}_\Omega - \dual{\phi}{v}_\Gamma &= \dual{f}{v}_\Omega +
       \dual{\phi_0}{v}_\Gamma, \\
       \dual{\psi}{(\tfrac12-\dlp)u}_\Gamma + \dual{\psi}{\slp\phi}_\Gamma
       &= \dual{\psi}{(\tfrac12-\dlp)u_0}_\Gamma,
    \end{split}
  \end{align}
  holds for all $\vv=(v,\psi) \in \HH$.

\subsection{Stabilization}\label{sec:jn_solv}
For the Johnson-N\'ed\'elec equations, we can apply similar techniques and derive similar results as for the
Bielak-MacCamy coupling, see Section~\ref{sec:bmc}.
For the sake of completeness, we state these results in the following.

We define the linear form $b_{\rm jn}(\cdot,\cdot)$ for $\uu,\vv\in\HH$ by
\begin{align*}
  \b{jn}(\uu,\vv) := \dual{\A\nabla u}{\nabla v}_\Omega &-
  \dual{\phi}{v}_\Gamma \\
  &+ \dual{\psi}{(\tfrac12-\dlp)u + \slp\phi}_\Gamma.
\end{align*}
Furthermore, we define linear functionals $\linfun_1$ and $\linfun_2$ by
\begin{align*}
  \linfun_1(v) &:= \dual{f}{v}_\Omega + \dual{\phi_0}{v}_\Gamma, \\
  \linfun_2(\psi) &:= \dual{\psi}{(\tfrac12-\dlp)u_0}_\Gamma
\end{align*}
for all $(v,\psi)\in\HH$.
Then, problem~\eqref{eq:jn} can be reformulated: Find $\uu \in\HH$ such that
\begin{align}\label{eq:jn_b}
  \b{jn}(\uu,\vv) = \L(\vv) := \linfun_1(v) + \linfun_2(\psi)
\end{align}
holds for all $\vv=(v,\psi)\in\HH$.
Moreover, the Galerkin discretization of~\eqref{eq:jn_b} reads: Find
$\UU_\ell \in \HH_\ell$ such that
\begin{align}\label{eq:jn_b_galerkin}
  \b{jn}(\UU_\ell,\VV_\ell) = \L(\VV_\ell)
\end{align}
holds for all $\VV_\ell \in\HH_\ell$, where $\HH_\ell = \XX_\ell\times\YY_\ell$ is a closed subspace of $\HH$. Similarly to Theorem~\ref{thm:equiv_b}, one proves the following result:

\begin{thm}\label{thm:equiv_b_jn}
With $\xi$ of Assumption~\ref{ass1}, the linear form
\begin{align}\label{eq:jn_bb:def}
  \bb{jn}(\uu,\vv) &:= \b{jn}(\uu,\vv) \\
  &+ \dual{\xi}{(\tfrac12 - \dlp)u + \slp\phi}_\Gamma \dual{\xi}{ (\tfrac12- \dlp)v + \slp\psi}_\Gamma, \nonumber
\end{align}
is equivalent to the linear form $\b{jn}(\cdot,\cdot)$ in the following sense:
The pair $\UU_\ell = (U_\ell,\Phi_\ell) \in\HH_\ell$ solves problem~\eqref{eq:jn_b_galerkin} if and only if it solves
  \begin{align}\label{eq:jn_bb}
  \begin{split}
    \bb{jn}(\UU_\ell,\VV_\ell) &= \linfun_1(V_\ell) + \linfun_2(\Psi_\ell) \\
    &+ \dual\xi{\linfun_2}_\Gamma\dual\xi{(\tfrac12-\dlp)V_\ell + \slp\Psi_\ell}_\Gamma
  \end{split}
  \end{align}
  for all $\VV_\ell = (V_\ell,\Psi_\ell) \in\HH_\ell$.\qed
\end{thm}

\subsection{Existence and uniqueness of solutions}\label{sec:jn_solv2}

We stress that there is a close link between the Johnson-N\'ed\'elec and
the Bielak-MacCamy coupling, since
\begin{align*}
  \b{jn}(\uu,\uu) = \b{bmc}(\uu,\uu) \quad\text{for all } \uu\in\HH.
\end{align*}
This indicates that the analytical techniques to prove ellipticity of the two coupling methods are similar.

The stabilized bilinear form $\bb{jn}(\cdot,\cdot)$ of~\eqref{eq:jn_bb:def} induces a nonlinear operator $\BB{jn} : \HH
\rightarrow \HH^*$ by
\begin{align}
\begin{split}
  \dual{\BB{jn}\uu}{\vv} &:= \bb{jn}(\uu,\vv)
\end{split}
\end{align}
for all $\uu,\vv\in\HH$. The following theorem states strong monotonicity of $\BB{jn}$ under the same assumptions as for Theorem~\ref{thm:bmc_ell}. Instead of Lemma~\ref{lemma:equiv_norm}, we need the following result: Under Assumption~\ref{ass1}, the definition
\begin{align}\label{eq:equiv_norm_jn}
\begin{split}
    \enorm{\uu}^2 &:= \norm{\nabla u}{L^2(\Omega)}^2 + \dual\phi{\slp\phi}_\Gamma \\
    &\qquad+
    \lvert \dual\xi{(\tfrac12-\dlp)u + \slp\phi}_\Gamma\rvert^2
\end{split}
\end{align}
for $\uu=(u,\phi)\in\HH$ provides an equivalent norm on $\HH$. The proof is achieved by a Rellich compactness argument as in the proof of Lemma~\ref{lemma:equiv_norm}.

\begin{thm}\label{thm:bmc_ell_jn}
Under Assumption~\ref{ass1} and provided that the ellipticity constant
$\c{ellA}$ of $\A$ fulfills $\c{ellA}>1/4$, the operator $\BB{jn}$ is
strongly monotone and Lipschitz continuous.\qed
\end{thm}

As above, Theorem~\ref{thm:equiv_b_jn} and standard theory on strong\-ly monotone operators now prove the following corollary.

\begin{cor}
Under the assumptions of Theorem~\ref{thm:bmc_ell_jn}, the Johnson-N\'ed\'elec coupling~\eqref{eq:jn_b} and its Galerkin discretization~\eqref{eq:jn_b_galerkin}
admit unique solutions $\uu\in\HH$ resp.\ $\UU_\ell\in\HH_\ell$. Moreover, there
holds quasi-optimality~\eqref{eq:cea}. \qed
\end{cor}
\subsection{Residual-based error estimator}\label{sec:jn_errest}
We recall a residual-based error estimator for the Johnson-N\'ed\'elec coupling.
The proof of the following result can be achieved with similar techniques as in
the proof of Theorem~\ref{thm:bmc_est}. The linear case for $d=2$ is found in~\cite{aposterjn}.
\begin{thm}\label{thm:jn_est}
  Suppose that $\uu \in\HH$ is the unique solution of the Johnson-N\'ed\'elec
  coupling~\eqref{eq:jn} and
  $\UU_\ell \in \HH_\ell = \SS^1(\TT_\ell) \times
  \PP^0(\EE_\ell^\Gamma)$
  is its Galerkin approximation~\eqref{eq:jn_b_galerkin}.
  Then, there holds reliability in the sense of
  \begin{align*}
    \c{rel}^{-2} \nnorm{\uu-\UU_\ell}^2 \leq \eta_\ell^2:=
    \sum_{T\in\TT_\ell} \eta_\ell(T)^2 +
    \sum_{E\in\EE_\ell^\Omega \cup \EE_\ell^\Gamma} \eta_\ell(E)^2.
  \end{align*}
  The volume contributions for $\TT_\ell$ and $\EE_\ell^\Omega$ are the same as
  above, cf.~\eqref{eq:bmc_errest_vol1} and~\eqref{eq:bmc_errest_vol2}, but are
  denoted by $\eta_\ell(T)$ resp. $\eta_\ell(E)$.
  The boundary contributions read
  \begin{align*}
    \eta_\ell(E)^2 &= h_E\norm{\phi_0 + \Phi_\ell -
      \A\nabla U_\ell\cdot\NV}{L^2(E)}^2 \\
      &\hspace{1cm}+
      h_E\norm{\nabla_\Gamma( (\tfrac12-\dlp)(u_0-U_\ell)
      -\slp\Phi_\ell)}{L^2(E)}^2
  \end{align*}
  for $E\in\EE_\ell^\Gamma$.
  The constant $\c{rel} >0$ depends only on $\Omega$, $\Gamma$
  and the $\gamma$-shape regularity of $\TT_\ell$ and $\EE_\ell^\Gamma$.\qed
\end{thm}

\section{Costabel's symmetric coupling}
\label{sec:sym}
In this section, we treat the symmetric FEM-BEM coupling method, which first
appeared in~\cite{costabel}. First, we
consider the continuous and discrete formulation.
Afterwards, we investigate well-posed\-ness of the coupling equations, where we
give a simpler proof of unique solvability than in the pioneering work~\cite{cs1995}. Finally, we recall a residual-based error estimator from the latter work.

\subsection{Derivation of symmetric coupling}\label{sec:sym_form}
We start from the Johnson-N\'ed\'elec coupling~\eqref{eq:jn} and modify the first equation: With the ansatz~\eqref{eq:jn_rep_form} for $u^{\rm ext}$,
we use the second Calder\'on identity
\begin{align}
  \phi = \partial_\NV^{\rm ext} u^{\rm ext} = -(\dlp^\dagger-\tfrac12)\phi -
  \hyp u^{\rm ext}.
\end{align}
The trace jump condition~\eqref{eq:strongform:trace} eliminates $u^{\rm ext}$ and gives
\begin{align}
  \phi = -(\dlp^\dagger-\tfrac12)\phi -
  \hyp u + \hyp u_0.
\end{align}
We plug this identity into the first equation of~\eqref{eq:jn} and move $\hyp u_0$ to the right-hand side. Altogether, the variational formulation of the symmetric coupling reads as follows:
Find $\uu=(u,\phi) \in \HH$ such that
  \begin{align}
    \begin{split}\label{eq:sym}
      &\dual{\A\nabla u}{\nabla v}_\Omega + \dual{(\dlp^\dagger-\tfrac12)\phi}{v}_\Gamma
      + \dual{\hyp u}{v}_\Gamma \\
      &\hspace{3cm}= \dual{f}{v}_\Omega + \dual{\phi_0 + \hyp u_0}{v}_\Gamma, \\
      &\dual{\psi}{(\tfrac12-\dlp)u}_\Gamma + \dual{\psi}{\slp\phi}_\Gamma = \dual{\psi}{(\tfrac12-\dlp)u_0}_\Gamma
    \end{split}
  \end{align}
hold for all $\vv=(v,\psi) \in \HH$. Note that the second equation of~\eqref{eq:sym} is just the same as for the Johnson-N\'ed\'elec coupling~\eqref{eq:jn}.

\subsection{Stabilization}\label{sec:sym_solv}
Unique solvability for the symmetric
coupling~\eqref{eq:sym} as well as for its
discretization
can be found in~\cite{cs1995} for our nonlinear model
problem~\eqref{eq:strongform}. Nevertheless, here we present a much simplified proof of this result.
We proceed as before and define the linear form $\b{s}(\cdot,\cdot)$ for all $\uu,\vv \in\HH$ by
\begin{align}
\begin{split}
  \b{s}(\uu,\vv) &:= \dual{\A\nabla u}{\nabla v}_\Omega +
  \dual{(\dlp^\dagger-\tfrac12)\phi}{v}_\Gamma \\
  &\qquad+ \dual{\hyp u }{v}_\Gamma + \dual{\psi}{(\tfrac12-\dlp)u + \slp\phi}_\Gamma.
\end{split}
\end{align}
Moreover, we define linear functionals $\linfun_1$ and $\linfun_2$ by
\begin{align*}
  \linfun_1(v) &:= \dual{f}{v}_\Omega + \dual{\phi_0 + \hyp u_0}{v}_\Gamma, \\
  \linfun_2(\psi) &:= \dual{\psi}{(\tfrac12-\dlp)u_0}_\Gamma.
\end{align*}
Then, problem~\eqref{eq:sym} can equivalently be stated as: Find $\uu\in\HH$ such
that
\begin{align}\label{eq:sym_b}
  \b{s}(\uu,\vv) = \L(\vv) := \linfun_1(v) + \linfun_2(\psi)
\end{align}
holds for all $\vv=(v,\psi)\in\HH$. For the discretization, let $\HH_\ell = \XX_\ell\times\YY_\ell$ be a closed subspace of $\HH$. The Galerkin discretization of~\eqref{eq:sym_b} then reads: Find
$\UU_\ell\in\HH_\ell$ such that
\begin{align}\label{eq:sym_b_galerkin}
  \b{s}(\UU_\ell,\VV_\ell) = \L(\VV_\ell)
\end{align}
holds for all $\VV_\ell \in\HH_\ell$.

The following theorem states equivalence of the linear form $\b{s}(\cdot,\cdot)$ to some new stabilized form $\bb{s}(\cdot,\cdot)$ which will turn out to be elliptic.
The proof is achieved by
similar techniques as used for proving Theorem~\ref{thm:equiv_b} and is thus omitted.

\begin{thm}\label{thm:equiv_b_sym}
With $\xi$ of Assumption~\ref{ass1}, the linear form
\begin{align}\label{eq:sym_bb:def}
  \bb{s}(\uu,\vv) &:= \b{s}(\uu,\vv) \\
  &\quad+ \dual{\xi}{\slp\phi + (\tfrac12 - \dlp)u}_\Gamma \dual{\xi}{\slp\psi
  + (\tfrac12- \dlp)v}_\Gamma\nonumber
\end{align}
is equivalent to the linear form $\b{s}(\cdot,\cdot)$ in the following sense:
The pair $\UU_\ell =
  (U_\ell,\Phi_\ell) \in\HH_\ell$ solves problem~\eqref{eq:sym_b_galerkin} if and only if it solves
  \begin{align}\label{eq:sym_bb}
  \begin{split}
    \bb{s}(\UU_\ell,\VV_\ell) &= \L(\VV_\ell) \\
    &+ \dual\xi{\linfun_2}_\Gamma\dual\xi{(\tfrac12-\dlp)V_\ell + \slp\Psi_\ell}_\Gamma
  \end{split}
  \end{align}
  for all $\VV_\ell = (V_\ell,\Psi_\ell) \in\HH_\ell$.\qed
\end{thm}

\subsection{Existence and uniqueness of solutions}\label{sec:sym_solv2}
The stabilized bilinear form $\bb{s}(\cdot,\cdot)$ from~\eqref{eq:sym_bb:def} induces a nonlinear operator $\BB{s} : \HH
\rightarrow \HH^*$ by
\begin{align}
\begin{split}
  \dual{\BB{s}\uu}{\vv} &:= \bb{s}(\uu,\vv)
\end{split}
\end{align}
for all $\uu,\vv\in\HH$. The following theorem states strong monotonicity of $\BB{s}$ \emph{without} any further restriction on the ellipticity constant $\c{ellA}>0$ of the nonlinearity $\A$.

\begin{thm}
Under Assumption~\ref{ass1}, the operator $\BB{s}$ is strongly monotone and Lipschitz continuous.
\end{thm}

\begin{proof}
Lipschitz continuity of $\BB{s}$ simply follows from the Lip\-schitz continuity of
$\A$ and the continuity of the boundary integral operators.

It thus only remains to show ellipticity of $\BB{s}$. We write $\ww = \uu-\vv = (w,\chi)$.
Recall that the hypersingular operator $\hyp$ is positive semi-definite. With the ellipticity constant $\c{ellA}>0$ of $\A$, we see
\begin{align}
\begin{split}
  \dual{&\BB{s}\uu - \BB{s}\vv}{\ww} \\
  &= \dual{\A\nabla u - \A\nabla v}{\nabla w}_\Omega +
  \dual{(\dlp^\dagger-\tfrac12)\chi}{w}_\Gamma \\
  &\quad+ \dual{\hyp w}{w}_\Gamma -\dual{\chi}{(\dlp-\tfrac12)w}_\Gamma +
  \dual{\chi}{\slp\chi}_\Gamma \\
  &\quad+ |\dual{\xi}{\slp\chi + (\tfrac12-\dlp)w}_\Gamma|^2 \\
  &\geq \c{ellA} \norm{\nabla w}{L^2(\Omega)}^2 +
  \dual{\chi}{\slp\chi}_\Gamma \\
  &\quad+ |\dual{\xi}{\slp\chi + (\tfrac12-\dlp)w}_\Gamma|^2 \\
  &\geq \min\{\c{ellA},1\} \big(\norm{\nabla w}{L^2(\Omega)}^2 +
  \dual{\chi}{\slp\chi}_\Gamma \\
  &\quad+ |\dual{\xi}{\slp\chi + (\tfrac12-\dlp)w}_\Gamma|^2 \big).
\end{split}
\end{align}
Norm equivalence for the norm $\enorm\cdot$ defined in~\eqref{eq:equiv_norm_jn} thus concludes
the proof of
ellipticity of the operator $\BB{s}$.\qed
\end{proof}

As above, Theorem~\ref{thm:equiv_b_sym} and standard theory on strong\-ly monotone operators prove the following corollary.

\begin{cor}
Under Assumption~\ref{ass1}, the symmetric coupling~\eqref{eq:sym_b} and its Galerkin discretization~\eqref{eq:sym_b_galerkin}
admit\linebreak unique solutions $\uu\in\HH$ and $\UU_\ell\in\HH_\ell$,
respectively. Moreover, there holds the
quasi-optimality~\eqref{eq:cea}. \qed
\end{cor}

\begin{rem} (i)
  Note that there is no restriction on the monotonicity constant $\c{ellA}>0$ of
  $\A$ as is the case for the Bielak-MacCamy and the Johnson-N\'ed\'elec coupling.
  \\ (ii)
  In contrast to~\cite{cs1995}, our analysis does not need the mesh-size
  $h_\ell$ to be
  sufficiently small to prove ellipticity of the coupling equations.
\end{rem}

\subsection{Residual-based error estimator}\label{sec:sym_errest}
The proof of the following result can be found in~\cite{cs1995} for $d=2$ and is achieved by similar techniques as in the proof of Theorem~\ref{thm:bmc_est}. It is therefore omitted.

\begin{thm}\label{thm:sym_est}
  Suppose that $\uu \in\HH$ is the unique solution of the symmetric
  coupling~\eqref{eq:sym_b} and
  $\UU_\ell \in \HH_\ell = \SS^1(\Omega) \times
  \PP^0(\EE_\ell^\Gamma)$ is its Galerkin approximation~\eqref{eq:sym_b_galerkin}.
  Then, there holds reliability in the sense of
  \begin{align*}
    \c{rel}^{-2} \nnorm{\uu-\UU_\ell}^2 \leq \mu_\ell^2 := \sum_{T\in\TT_\ell} \mu_\ell(T)^2 +
    \sum_{E\in\EE_\ell^\Omega \cup \EE_\ell^\Gamma} \mu_\ell(E)^2.
  \end{align*}
  The volume contributions for $\TT_\ell$ and $\EE_\ell^\Omega$ are the same as
  above, cf.~\eqref{eq:bmc_errest_vol1} and~\eqref{eq:bmc_errest_vol2}, but are
  denoted by $\mu_\ell(T)$ resp. $\mu_\ell(E)$.
  The boundary contributions read
  \begin{align*}
    \mu_\ell(E)^2 = &h_E\norm{\phi_0- \A\nabla U_\ell\cdot\NV+
    \hyp(u_0-U_\ell) \\
    &\hspace{4cm}-(\dlp^\dagger-\tfrac12)\Phi_\ell}{L^2(E)}^2 \\
       +&h_E\norm{\nabla_\Gamma( (\tfrac12-\dlp)(u_0-U_\ell)
      -\slp\Phi_\ell)}{L^2(E)}^2
  \end{align*}
  for $E\in\EE_\ell^\Gamma$.
  The constant $\c{rel} >0$ depends only on $\Omega$, $\Gamma$,
  and the $\gamma$-shape regularity of $\TT_\ell$ and $\EE_\ell^\Gamma$.\qed
\end{thm}

\section{Convergence of adaptive scheme}
\label{sec:convergence}
In this section, we state an adaptive algorithm for the three different coupling
methods and prove convergence of the adaptive scheme for the Bielak-MacCamy
coupling.

\subsection{Adaptive algorithm \& Mesh-refinement}\label{sec:alg_meshref}
In the following, we fix one particular coupling and let
$\UU_\ell\in\HH_\ell:=\SS^1(\TT_\ell)\times\PP^0(\EE_\ell^\Gamma)$ denote the corresponding Galerkin solution, i.e.\ $\UU_\ell$ solves either problem~\eqref{eq:bmc_b_galerkin}, \eqref{eq:jn_b_galerkin},
or~\eqref{eq:sym_b_galerkin}. By $\zeta_\ell$, we denote the corresponding
residual-based error estimator, cf.\
Sections~\ref{sec:bmc_errest},~\ref{sec:jn_errest}, and~\ref{sec:sym_errest}.
Under these assumptions, the standard adaptive scheme reads as follows:

\begin{alg}\label{alg:adaptive}
  \textsc{Input: } Initial triangulation $(\TT_0, \EE_0^\Gamma)$,
  counter $\ell:=0$, and adaptivity parameter $0<\theta<1$.
  \begin{itemize}
    \item[(i)] Compute discrete solution $\UU_\ell \in\HH_\ell$.
    \item[(ii)] Compute local refinement indicators $\zeta_\ell(\tau)$ for all
    $\tau \in \TT_\ell \cup \EE_\ell^\Omega \cup \EE_\ell^\Gamma$.
    \item[(iii)] Find (minimal) set $\MM_\ell \subseteq \TT_\ell \cup \EE_\ell^\Omega
    \cup \EE_\ell^\Gamma$ such that the error estimator $\zeta_\ell$ fulfills
    the \Doerfler~marking
    \begin{align}
      \theta \zeta_\ell^2 \leq \sum_{T\in\MM_\ell\cap\TT_\ell} \zeta_\ell(T)^2 +
      \sum_{E\in\MM_\ell\cap(\EE_\ell^\Omega \cup \EE_\ell^\Gamma)}
      \zeta_\ell(E)^2.
    \end{align}
    \item[(iv)] Obtain new mesh $(\TT_{\ell+1},\EE_{\ell+1}^\Gamma)$ by refining at
    least all elements of the set $\MM_\ell$.
    \item[(v)] Increase counter $\ell$ by one and goto (i).
  \end{itemize}
  \textsc{Output: } Sequence of Galerkin solutions $(\UU_\ell)_{\ell\in\N_0}$ and
  sequence of error estimators $(\zeta_\ell)_{\ell\in\N_0}$.\qed
\end{alg}

To prove convergence of the adaptive algorithm, we need certain assumptions on
the mesh-refining strategy used in step \rm{(iv)}.
\begin{ass}[Mesh-refinement]\label{ass:meshref}
  Let $(\TT_\ell,\EE_\ell^\Gamma)$ be a sequence of regular triangulations, where
  $\TT_{\ell+1}$ is obtained from $\TT_\ell$ and $\EE_{\ell+1}^\Gamma$ from
  $\EE_\ell^\Gamma$ by local mesh-refinement.
  Then, there are $\ell$-independent constants
  $0<\gamma<\infty$ and $0<q<1$ such that the following holds
  \begin{itemize}
    \item[$\bullet$] The triangulations $(\TT_\ell,\EE_\ell^\Gamma)$ are $\gamma$-shape regular.
    \item[$\bullet$] A refined element $T\in\TT_\ell$ resp.\ face $E\in\EE_\ell^\Gamma$ is the union
    of its sons $T'\in\TT_{\ell+1}$ resp.\ $E'\in\EE_{\ell+1}^\Gamma$.
    \item[$\bullet$] The sons $T' \in \TT_{\ell+1}$ of a refined element $T\in\TT_\ell$
    satisfy $|T'| \leq q |T|$. The sons $E'\in\EE_{\ell+1}^\Omega$ of a refined
    interior face $E\in\EE_\ell^\Omega$ satisfy $|E'|\leq q |E|$.
    \item[$\bullet$] The sons $E' \in\EE_{\ell+1}^\Gamma$ of a refined boundary face
    $E\in\EE_\ell^\Gamma$ satisfy $|E'|\leq q |E|$.\qed
  \end{itemize}
\end{ass}

We note that Assumption~\ref{ass:meshref} implies nestedness of the discrete
spaces, i.e. $\HH_\ell \subseteq \HH_{\ell+1}$.
Moreover, we get
\begin{align}\label{eq:mred_T}
  h_{\ell+1}|_T \leq \begin{cases} q^{1/d} h_\ell|_T &\quad\text{for }
  T\in\TT_\ell\backslash\TT_{\ell+1}, \\
  h_\ell|_T &\quad\text{for } T\in\TT_{\ell+1}\cap\TT_\ell,
  \end{cases}
\end{align}
\begin{align}\label{eq:mred_Eint}
  h_{\ell+1}|_E \leq \begin{cases} q^{1/(d-1)} h_\ell|_E &\quad\text{for }
  E\in\EE_\ell^\Omega\backslash\EE_{\ell+1}^\Omega, \\
  h_\ell|_E &\quad\text{for } E\in\EE_{\ell+1}^\Omega\cap\EE_\ell^\Omega,
  \end{cases}
\end{align}
and
\begin{align}\label{eq:mred_E}
  h_{\ell+1}|_E \leq \begin{cases} q^{1/(d-1)} h_\ell|_E &\quad\text{for }
  E\in\EE_\ell^\Gamma\backslash\EE_{\ell+1}^\Gamma, \\
  h_\ell|_E &\quad\text{for } E\in\EE_{\ell+1}^\Gamma\cap\EE_\ell^\Gamma.
  \end{cases}
\end{align}
We stress that 2D and 3D newest vertex bisection, e.g.~\cite{stevenson},
satisfy Assumption~\ref{ass:meshref}.

\subsection{Convergence of adaptive algorithm}\label{sec:conv_alg}
In this subsection, we show convergence of the adaptive scheme presented above.
With the quasi-optimality~\eqref{eq:cea} at hand, one can prove
convergence of the sequence $\UU_\ell$ to a limit $\UU_\infty \in \HH_\infty$, where
$\HH_\infty$ is the closure of $\bigcup_{\ell=0}^\infty \HH_\ell$ and where $\UU_\infty$ is the unique Galerkin solution in $\HH_\infty$.
In particular, this implies $\nnorm{\UU_{\ell+1}-\UU_\ell}\rightarrow 0$ as $\ell\to\infty$.
A priori, it is unclear if $\uu = \UU_\infty$, since adaptive
mesh-refining may lead to meshes, where the local mesh-size $h_\ell$
does not tend to zero in $L^\infty(\Omega)$ and $L^\infty(\Gamma)$.

To prove convergence, $\uu=\UU_\infty$, we follow the estimator reduction
principle of~\cite{estconv} as is done in~\cite{afp} for $(h-h/2)$-type error estimators and the symmetric coupling: By use of the D\"orfler marking and the mesh-refinement strategy, we verify a perturbed contraction estimate
\begin{align}\label{eq:estredest}
  \zeta_{\ell+1}^2
  \leq \kappa\, \zeta_\ell^2
  + \c{reduction}\nnorm{\UU_{\ell+1}- \UU_\ell}^2,
\end{align}
with certain $\ell$-independent constants $0<\kappa<1$ and $\c{reduction}>0$.
Together with the a~priori convergence, elementary calculus proves
$\zeta_\ell\rightarrow 0$ as $\ell\to\infty$. Finally, reliability of
the error estimator $\zeta_\ell$ proves $\UU_\ell\rightarrow \uu$ as $\ell\rightarrow\infty$.

One main ingredient for the proof of the estimator reduction estimate~\eqref{eq:estredest} are the following
novel inverse-type estimates from~\cite{invest3D} for the boundary integral
operators involved.

\begin{lem}\label{lemma:invest}
 There exists a constant $C_{\rm inv}>0$ such that the estimates
 \begin{align*}
   \norm{h_\ell^{1/2}\nabla_\Gamma \dlp V_\ell}{L^2(\Gamma)} &\leq C_{\rm inv}
   \norm{V_\ell}{H^{1/2}(\Gamma)}, \\
   \norm{h_\ell^{1/2}\hyp V_\ell}{L^2(\Gamma)} &\leq C_{\rm inv}
   \norm{V_\ell}{H^{1/2}(\Gamma)}, \\
   \norm{h_\ell^{1/2}\nabla_\Gamma \slp\Psi_\ell}{L^2(\Gamma)} &\leq C_{\rm inv}
   \norm{\Psi_\ell}{H^{-1/2}(\Gamma)}, \\
   \norm{h_\ell^{1/2}\dlp^\dagger \Psi_\ell}{L^2(\Gamma)} &\leq C_{\rm inv}
   \norm{\Psi_\ell}{H^{-1/2}(\Gamma)},
 \end{align*}
 hold for all discrete functions $V_\ell\in\SS^1(\TT_\ell)$ and $\Psi_\ell\in
 \PP^0(\EE_\ell^\Gamma)$. The constant $C_{\rm inv}$ depends only on $\Gamma$
 and the $\gamma$-shape regularity of $\EE_\ell^\Gamma$.\qed
\end{lem}

With this lemma, we can prove the main result of this section, which states convergence of the adaptive Bielak-MacCamy coupling. The same result also holds for the adaptive Johnson-N\'ed\'elec coupling as well as the adaptive symmetric coupling. The latter is treated in~\cite{invest3D}.

\begin{thm}\label{thm:estred}
Let $\uu\in\HH$ be the solution of the Bielak-MacCamy equations~\eqref{eq:bmc}. Let $\UU_\ell$ be the sequence of Ga\-ler\-kin solutions generated by Algorithm~\ref{alg:adaptive}. Then, the sequence of corresponding error estimators $\zeta_\ell = \rho_\ell$ fulfills the estimator reduction estimate~\eqref{eq:estredest}. In particular, this implies convergence, i.e.
$\displaystyle \lim\limits_{\ell\rightarrow\infty} \nnorm{\uu-\UU_\ell} = 0 =
  \lim\limits_{\ell\rightarrow\infty} \rho_\ell$.
\end{thm}

\begin{proof}
To abbreviate notations, we define
\begin{align*}
  \rho_\ell(\FF)^2 := \sum_{T\in \FF\cap\TT_\ell} \rho_\ell(T)^2
  + \sum_{E\in \FF \cap (\EE_\ell^\Omega\cup \EE_\ell^\Gamma)} \rho_\ell(E)^2
\end{align*}
for any subset $\FF\subseteq \TT_\ell \cup \EE_\ell^\Omega \cup
\EE_\ell^\Gamma$.
We aim to estimate each contribution of the error estimator
\begin{align}\label{eq:rho_cont}
  \rho_{\ell+1}^2 = \rho_{\ell+1}(\TT_{\ell+1})^2 + \rho_{\ell+1}
  (\EE_{\ell+1}^\Omega)^2 + \rho_{\ell+1}(\EE_{\ell+1}^\Gamma)^2.
\end{align}

\textbf{Step~1. }
By use of~\eqref{eq:mred_T}, we can estimate the first volume
contribution $\rho_{\ell+1}(\TT_{\ell+1})^2$ by
\begin{align}\label{eq:rho_vol}
\begin{split}
  \rho_{\ell+1}(\TT_{\ell+1})^2 &= \norm{h_{\ell+1} f}{L^2(\Omega)}^2 =
  \sum_{T\in\TT_\ell} \norm{h_{\ell+1} f}{L^2(T)}^2 \\
  &\leq \sum_{T\in\TT_\ell\cap\TT_{\ell+1}} \norm{h_\ell f}{L^2(T)}^2 \\
  &\hspace{0.5cm}+q^{2/d}\sum_{T\in\TT_\ell\backslash\TT_{\ell+1}}
  \norm{h_\ell f}{L^2(T)}^2 \\
  & = \sum_{T\in\TT_\ell} \norm{h_\ell f}{L^2(T)}^2 \\
  &\hspace{0.5cm}- (1- q^{2/d})
  \sum_{T\in\TT_\ell\backslash\TT_{\ell+1}} \norm{h_\ell f}{L^2(T)}^2 \\
  &= \rho_\ell(\TT_\ell)^2 - (1-q^{2/d}) \rho_\ell(\TT_\ell\backslash
  \TT_{\ell+1})^2,
\end{split}
\end{align}
where $q$ denotes the mesh-reduction constant from Assumption~\ref{ass:meshref}.

\textbf{Step~2. }
To estimate the second term in~\eqref{eq:rho_cont}, we use the Young
inequality $(a+b)^2\leq (1+\delta)a^2 +
(1+\delta^{-1})b^2$ for arbitrary $a,b\in\R$ and $\delta>0$.
With this and the triangle inequality, we see
\begin{align}\label{eq:rho_intedge1}
\begin{split}
 &\rho_{\ell+1}(\EE_{\ell+1}^\Omega)^2 \\
 &\quad= \sum_{E\in\EE_{\ell+1}^\Omega}
 \norm{h_{\ell+1}^{1/2}[\A\nabla U_{\ell+1} \cdot\NV]}{L^2(E)}^2 \\
 &\quad\leq (1+\delta) \sum_{E\in\EE_{\ell+1}^\Omega} \norm{h_{\ell+1}^{1/2}
 [\A\nabla U_\ell\cdot\NV]}{L^2(E)}^2 \\
 &\quad\quad+(1+\delta^{-1})
 \\&\quad\quad\quad\quad
 \sum_{E\in\EE_{\ell+1}^\Omega}
 \norm{h_{\ell+1}^{1/2} [(\A\nabla U_{\ell+1}-\A\nabla U_\ell)\cdot\NV]}
 {L^2(E)}^2.\hspace*{-1cm}
\end{split}
\end{align}
Recall that $\A\nabla U_\ell\in(\PP^0(\TT_\ell))^2$. Therefore,
$[\A\nabla U_\ell\cdot\NV]|_{E_1'} = [\A\nabla U_\ell\cdot\NV]|_{E_2'}$ for all sons
$E_1',E_2'\in\EE_{\ell+1}^\Omega$ of a refined
element $E\in\EE_\ell^\Omega$.
Furthermore, jumps over new interior faces vanish.
With this and~\eqref{eq:mred_Eint}, we can estimate the first sum on the right-hand side of
\eqref{eq:rho_intedge1} by
\begin{align*}
  \sum_{E\in\EE_{\ell+1}^\Omega} &\norm{h_{\ell+1}^{1/2}
  [\A\nabla U_\ell\cdot\NV]}{L^2(E)}^2 \\
  &\leq \sum_{E\in\EE_\ell^\Omega \cap \EE_{\ell+1}^\Omega}
  \norm{h_\ell^{1/2} [\A\nabla U_\ell\cdot\NV]}{L^2(E)}^2 \\
  &\hspace{0.5cm}+ q^{1/(d-1)} \sum_{E\in\EE_\ell^\Omega \backslash \EE_{\ell+1}^\Omega}
  \norm{h_\ell^{1/2}[\A\nabla U_\ell\cdot\NV]}{L^2(E)}^2 \\
  &= \rho_\ell(\EE_\ell^\Omega)^2 - (1-q^{1/(d-1)}) \rho_\ell^\Omega(\EE_\ell\backslash
  \EE_{\ell+1}^\Omega)^2.
\end{align*}
The summands in the second sum on the right-hand side of~\eqref{eq:rho_intedge1}
are bounded from above by use
of the pointwise
Lipschitz continuity of $\A$ and a scaling argument, i.e.
\begin{align}\label{eq:rho_intedge2}
\begin{split}
  \norm{h_{\ell+1}^{1/2}[(\A\nabla U_{\ell+1} &-
  \A\nabla U_\ell)\cdot\NV]}{L^2(E)}^2 \\
  \hspace{1cm}&\lesssim \c{lipA}^2 \norm{\nabla(U_{\ell+1}-U_\ell)}
  {L^2(\omega_{\ell+1,E})}^2.
\end{split}
\end{align}
We sum~\eqref{eq:rho_intedge2} over all interior faces and get
\begin{align}\label{eq:rho_intedge3}
\begin{split}
  &\rho_{\ell+1}(\EE_{\ell+1}^\Omega)^2 \\
  &\;\,\leq (1+\delta)
  \big(\rho_\ell(\EE_\ell^\Omega)^2
  - (1-q^{1/(d-1)})\rho_\ell( \EE_\ell^\Omega
  \backslash \EE_{\ell+1}^\Omega)^2\big) \\
  &\;\quad\quad+ (1+\delta^{-1}) C \norm{\nabla(U_{\ell+1}-U_\ell)}{L^2(\Omega)}^2,
\end{split}
\end{align}
where the constant $C>0$ depends only on the $\gamma$-shape regularity of
$\TT_{\ell+1}$ and the Lipschitz constant of $\A$.

\textbf{Step~3. }
We consider the boundary contributions of
$\rho_{\ell+1}^2$ and introduce the splitting $\rho_\ell(E)^2=
\rho_\ell^{(1)}(E)^2 + \rho_\ell^{(2)}(E)^2$, where
\begin{align*}
  \rho_\ell^{(1)}(E) &= \norm{h_\ell^{1/2} (\phi_0 + (\dlp^\dagger -
  \tfrac12)\Phi_\ell - \A\nabla U_\ell\cdot\NV)}{L^2(E)}, \\
  \rho_\ell^{(2)}(E) &=\norm{h_\ell^{1/2} \nabla_\Gamma (U_\ell-u_0 -
  \slp\Phi_\ell )}{L^2(E)}
\end{align*}
for $E\in\EE_\ell^\Gamma$.
Again we use the triangle inequality and estimate $\rho_{\ell+1}^{(1)}(E)$ by
\begin{align*}
  \rho_{\ell+1}^{(1)}(E)
  \leq &\norm{h_{\ell+1}^{1/2}(\phi_0 +
  (\dlp^\dagger - \tfrac12)\Phi_\ell - \A\nabla U_\ell \cdot \NV)}{L^2(E)} \\
  &+ \norm{h_{\ell+1}^{1/2} \dlp^\dagger (\Phi_{\ell+1}-\Phi_\ell)}{L^2(E)} \\
  &+ \frac12\norm{h_{\ell+1}^{1/2} (\Phi_{\ell+1}-\Phi_\ell)}{L^2(E)} \\
  &+ \norm{h_{\ell+1}^{1/2} (\A\nabla U_{\ell+1}- \A\nabla
  U_\ell)\cdot\NV}{L^2(E)}
\end{align*}
for $E\in\EE_{\ell+1}^\Gamma$.
Summing $\rho_{\ell+1}^{(1)}(E)^2$ over all boundary faces and applying the Young inequality, we end up with
\begin{align}\label{eq:rho_boundedge5}
\begin{split}
  &\rho_{\ell+1}^{(1)}(\EE_{\ell+1}^\Gamma)^2 \\
  &\leq (1+\delta)
  \sum_{E\in\EE_{\ell+1}^\Gamma} \norm{h_{\ell+1}^{1/2} (\phi_0-(\dlp^\dagger -
  \tfrac12)\Phi_\ell - \A\nabla U_\ell \cdot \NV)}{L^2(E)}^2 \\
  &+3(1+\delta^{-1}) \sum_{E\in\EE_{\ell+1}^\Gamma} \norm{h_{\ell+1}^{1/2}
  \dlp^\dagger (\Phi_{\ell+1}-\Phi_\ell)}{L^2(E)}^2 \\
  &+3(1+\delta^{-1}) \sum_{E\in\EE_{\ell+1}^\Gamma}
  \frac14\norm{h_{\ell+1}^{1/2} (\Phi_{\ell+1}-\Phi_\ell)}{L^2(E)}^2 \\
  &+3(1+\delta^{-1}) \sum_{E\in\EE_{\ell+1}^\Gamma}
  \norm{h_{\ell+1}^{1/2} (\A\nabla U_{\ell+1}- \A\nabla
  U_\ell)\cdot\NV}{L^2(E)}^2,
\end{split}
\end{align}
where the factor $3$ stems from the inequality $(\sum_{i=1}^n x_i)^2 \leq n
\sum_{i=1}^n x_i^2$. With the help of~\eqref{eq:mred_E}, we argue similarly as
before and estimate the first
sum in~\eqref{eq:rho_boundedge5} by
\begin{align*}
  &\sum_{E\in\EE_{\ell+1}^\Gamma} \norm{h_{\ell+1}^{1/2} (\phi_0-(\dlp^\dagger -
  \tfrac12)\Phi_\ell - \A\nabla U_\ell \cdot \NV)}{L^2(E)}^2 \\
  &\quad\leq \sum_{E\in\EE_\ell^\Gamma \cap \EE_{\ell+1}^\Gamma}
  \norm{h_\ell^{1/2} (\phi_0-(\dlp^\dagger -
  \tfrac12)\Phi_\ell - \A\nabla U_\ell \cdot \NV)}{L^2(E)}^2 \\
  &\quad \quad\quad\quad+ q^{1/(d-1)}\rho_\ell^{(1)}(\EE_\ell^\Gamma \backslash \EE_{\ell+1}^\Gamma)^2\\
  &\quad= \rho_\ell^{(1)}(\EE_\ell^\Gamma)^2 - (1-q^{1/(d-1)})
  \rho_\ell^{(1)}(\EE_\ell^\Gamma \backslash \EE_{\ell+1}^\Gamma)^2.
\end{align*}
An inverse-type estimate
from Lemma~\ref{lemma:invest} can be applied to the second sum of~\eqref{eq:rho_boundedge5}.
This yields
\begin{align*}
  \sum_{E\in\EE_{\ell+1}^\Gamma} \norm{h_{\ell+1}^{1/2} &
  \dlp^\dagger (\Phi_{\ell+1}-\Phi_\ell)}{L^2(E)}^2 \\
  &= \norm{h_{\ell+1}^{1/2}
  \dlp^\dagger (\Phi_{\ell+1}-\Phi_\ell)}{L^2(\Gamma)}^2 \\
  &\lesssim \norm{\Phi_{\ell+1}-\Phi_\ell}{H^{-1/2}(\Gamma)}^2.
\end{align*}
Here, the hidden constant depends only on $\Gamma$ and the $\gamma$-shape
regularity of $\EE_{\ell+1}^\Gamma$.
For the third sum, we use an inverse estimate
from~\cite[Theorem~3.6]{ghs} to see
\begin{align*}
  \sum_{E\in\EE_{\ell+1}^\Gamma}\norm{h_{\ell+1}^{1/2}
  (\Phi_{\ell+1}-\Phi_\ell)}{L^2(E)}^2
  &= \norm{h_{\ell+1}^{1/2} (\Phi_{\ell+1}-\Phi_\ell)}{L^2(\Gamma)}^2
  \\
  &\lesssim \norm{\Phi_{\ell+1}-\Phi_\ell}{H^{-1/2}(\Gamma)}^2,
\end{align*}
where, as before, the hidden constant depends only on $\Gamma$ and the
$\gamma$-shape regularity of $\EE_{\ell+1}^\Gamma$.
For the last term in~\eqref{eq:rho_boundedge5}, we again use
a scaling argument and infer
\begin{align*}
  \sum_{E\in\EE_{\ell+1}^\Gamma}
  \norm{h_{\ell+1}^{1/2}& (\A\nabla U_{\ell+1}- \A\nabla
  U_\ell)\cdot\NV}{L^2(E)}^2 \\
  &\lesssim \norm{\nabla
  (U_{\ell+1}-U_\ell)}{L^2(\Omega)}^2,
\end{align*}
where the hidden constant depends only on the Lip\-schitz constant of $\A$ and the
$\gamma$-shape regularity of $\TT_{\ell+1}$.
Altogether, we see
\begin{align}\label{eq:rho_boundedge10}
\begin{split}
  &\rho_{\ell+1}^{(1)}(\EE_{\ell+1}^\Gamma)^2 \\
  &\leq (1+\delta)
  \big(\rho_\ell^{(1)}(\EE_\ell^\Gamma)^2 - (1-q^{1/(d-1)})
  \rho_\ell^{(1)}(\EE_\ell^\Gamma \backslash \EE_{\ell+1}^\Gamma)^2\big) \\
  &\quad+(1+\delta^{-1}) C \big(\norm{\nabla(U_{\ell+1}-U_\ell)}{L^2(\Omega)}^2 \\
  &\quad\hspace{3cm}+
  \norm{\Phi_{\ell+1}-\Phi_\ell}{H^{-1/2}(\Gamma)}^2\big)
\end{split}
\end{align}
It remains to estimate the second boundary contribution $\rho_{\ell+1}^2$.
Arguing as before, we get
\begin{align}\label{eq:rho_boundedge11}
\begin{split}
  &\rho_{\ell+1}^{(2)}(\EE_{\ell+1}^\Gamma)^2 \\
  &\leq (1+\delta)
  \sum_{E\in\EE_{\ell+1}^\Gamma} \norm{h_{\ell+1}^{1/2} \nabla_\Gamma (U_\ell-
  u_0 - \slp\Phi_\ell )}{L^2(E)}^2 \\
  & \;\,+ 2(1+\delta^{-1}) \sum_{E\in\EE_{\ell+1}^\Gamma} \norm{h_{\ell+1}^{1/2}
  \nabla_\Gamma (U_{\ell+1} -U_\ell )}{L^2(E)}^2 \\
  & \;\,+ 2(1+\delta^{-1}) \sum_{E\in\EE_{\ell+1}^\Gamma} \norm{h_{\ell+1}^{1/2}
  \nabla_\Gamma \slp(\Phi_{\ell+1} -\Phi_\ell )}{L^2(E)}^2.
\end{split}
\end{align}
We use~\eqref{eq:mred_E} and estimate the first sum in
\eqref{eq:rho_boundedge11} by
\begin{align}\label{eq:rho_boundedge12}
\begin{split}
 & \sum_{E\in\EE_{\ell+1}^\Gamma} \norm{h_{\ell+1}^{1/2} \nabla_\Gamma (U_\ell-
  u_0 - \slp\Phi_\ell )}{L^2(E)}^2 \\
  &\quad\leq \sum_{E\in\EE_\ell^\Gamma\cap
  \EE_{\ell+1}^\Gamma} \norm{h_\ell^{1/2} \nabla_\Gamma (U_\ell-
  u_0 - \slp\Phi_\ell )}{L^2(E)}^2 \\
  &\quad\quad\quad\quad+q^{1/(d-1)} \rho_\ell^{(2)}(\EE_\ell^\Gamma \backslash \EE_{\ell+1}^\Gamma)^2\\
  &\quad= \rho_\ell^{(2)}(\EE_\ell^\Gamma)^2 - (1-q^{1/(d-1)})
  \rho_\ell^{(2)}(\EE_\ell^\Gamma \backslash \EE_{\ell+1}^\Gamma)^2.
\end{split}
\end{align}
For the second sum in~\eqref{eq:rho_boundedge11}, an inverse estimate
from~\cite[Proposition~3]{akp} can be applied, which yields
\begin{align}\label{eq:rho_boundedge13}
\begin{split}
  \sum_{E\in\EE_{\ell+1}^\Gamma}& \norm{h_{\ell+1}^{1/2}
  \nabla_\Gamma (U_{\ell+1} -U_\ell )}{L^2(E)}^2 \\
  &= \norm{h_{\ell+1}^{1/2} \nabla_\Gamma (U_{\ell+1} -
  U_\ell )}{L^2(\Gamma)}^2\\
  & \lesssim \norm{U_{\ell+1}-U_\ell}{H^{1/2}(\Gamma)}^2.
\end{split}
\end{align}
Here, the hidden constant depends only on $\Gamma$ and the $\gamma$-shape
regularity of $\EE_{\ell+1}^\Gamma$.
Finally, the inverse estimate for $\slp$ of Lemma~\ref{lemma:invest} proves
\begin{align}\label{eq:rho_boundedge14}
\begin{split}
  \sum_{E\in\EE_{\ell+1}^\Gamma}& \norm{h_{\ell+1}^{1/2}
  \nabla_\Gamma \slp(\Phi_{\ell+1} -\Phi_\ell )}{L^2(E)}^2 \\
  &= \norm{h_{\ell+1}^{1/2}\nabla_\Gamma \slp(\Phi_{\ell+1} -\Phi_\ell
  )}{L^2(\Gamma)}^2\\
  & \lesssim \norm{\Phi_{\ell+1}-\Phi_\ell}{H^{-1/2}(\Gamma)}^2
\end{split}
\end{align}
for the third sum in~\eqref{eq:rho_boundedge11}. Again, the hidden constant
depends only on $\Gamma$ and the $\gamma$-shape regularity of
$\EE_{\ell+1}^\Gamma$.
Combining~\eqref{eq:rho_boundedge11}, \eqref{eq:rho_boundedge12},
\eqref{eq:rho_boundedge13}, and \eqref{eq:rho_boundedge14} gives
\begin{align}\label{eq:rho_boundedge15}
\begin{split}
  &\rho_{\ell+1}^{(2)}(\EE_{\ell+1}^\Gamma)^2 \\
  &\leq (1+\delta)
  \big(\rho_\ell^{(2)}(\EE_\ell^\Gamma)^2 - (1-q^{1/(d-1)})
  \rho_\ell^{(2)}(\EE_\ell^\Gamma \backslash \EE_{\ell+1}^\Gamma)^2\big)\\
  &\quad\quad+(1+\delta^{-1}) C \big(\norm{U_{\ell+1}-U_\ell}{H^{1/2}(\Gamma)}^2 \\
  &\hspace{4cm} + \norm{\Phi_{\ell+1}-\Phi_\ell}{H^{-1/2}(\Gamma)}^2\big)
\end{split}
\end{align}

\textbf{Step~4. }
We combine the estimates~\eqref{eq:rho_vol}, \eqref{eq:rho_intedge3},
\eqref{eq:rho_boundedge10}, and \eqref{eq:rho_boundedge15}
for the different contributions of the estimator.
This results in
\begin{align*}
\begin{split}
  \rho_{\ell+1}^2 &= \rho_{\ell+1}(\TT_{\ell+1})^2 +
  \rho_{\ell+1}(\EE_{\ell+1}^\Omega)^2 +
  \rho_{\ell+1}(\EE_{\ell+1}^\Gamma)^2 \\
  &\leq (1+\delta)\Big(\rho_\ell(\TT_\ell)^2 + \rho_\ell(\EE_\ell^\Omega)^2 +
  \rho_\ell(\EE_\ell^\Gamma)^2 \\
  &\hspace{1cm} - (1-q^{1/(d-1)}) \big(\rho_\ell(\TT_\ell\backslash\TT_{\ell+1})^2
  \\&\hspace{2cm}+ \rho_\ell(\EE_\ell^\Omega\backslash\EE_{\ell+1}^\Omega)^2
  + \rho_\ell(\EE_\ell^\Gamma\backslash\EE_{\ell+1}^\Gamma)^2\big) \Big) \\
  &\quad\;+ (1+\delta^{-1}) C' \big( \norm{\nabla(U_{\ell+1} -U_\ell)}{L^2(\Omega)}^2 \\
  &\hspace{0.8cm}+
  \norm{U_{\ell+1}-U_\ell}{H^{1/2}(\Gamma)}^2 +
  \norm{\Phi_{\ell+1}-\Phi_\ell}{H^{-1/2}(\Gamma)}^2\big),
\end{split}
\end{align*}
where we have used $q^{2/d}\leq q^{1/(d-1)}$, since $q<1$. We define $\II_\ell :=
\TT_\ell \cup \EE_\ell^\Omega \cup \EE_\ell^\Gamma$.
Then, the norm equivalence
\begin{align*}
&\norm{\nabla v}{L^2(\Omega)}^2 +
\norm{v}{H^{1/2}(\Gamma)}^2 \simeq \norm{v}{H^1(\Omega)}^2 \text{ for } v\in
H^1(\Omega)
\end{align*}
yields
\begin{align}\label{eq:rho_est}
\begin{split}
  \rho_{\ell+1}^2 &\leq (1+\delta)\big(\rho_\ell^2 - (1-q^{1/(d-1)})
  \rho_\ell(\II_\ell\backslash \II_{\ell+1})^2\big) \\
  &\hspace{1cm}+ (1+\delta^{-1}) C
  \nnorm{\UU_{\ell+1}-\UU_\ell}^2.
\end{split}
\end{align}
The constant $C>0$ depends on $\Omega$, the $\gamma$-shape regularity of
$\TT_{\ell+1}$ and $\EE_{\ell+1}^\Gamma$,
and the
Lipschitz constant of~$\A$.

\textbf{Step~5. } Recall the \Doerfler~marking
\begin{align*}
  \theta\rho_\ell^2 \leq \rho_\ell(\MM_\ell)^2 \leq \rho_\ell(\II_\ell
  \backslash\II_{\ell+1})^2,
\end{align*}
where $\MM_\ell \subseteq \II_\ell\backslash\II_{\ell+1}$ denotes the set of
marked elements.
Incorporating the last inequality in~\eqref{eq:rho_est} gives
\begin{align}\label{eq:rho_est2}
\begin{split}
  \rho_{\ell+1}^2 &\leq (1+\delta) (\rho_\ell^2 -(1-q^{1/(d-1)})\theta\rho_\ell^2)
  \\ &\hspace{1cm}+(1+\delta^{-1}) C \nnorm{\UU_{\ell+1}-\UU_\ell}^2 \\
  &= (1+\delta)(1-\theta (1-q^{1/(d-1)}))\rho_\ell^2
  \\ &\hspace{1cm}+(1+\delta^{-1}) C \nnorm{\UU_{\ell+1}-\UU_\ell}^2.
\end{split}
\end{align}
Since $0<q^{1/(d-1)}<1$ and $0<\theta <1$, we have \ $0<1-\theta(1-q^{1/(d-1)}) < 1$.
Choosing $\delta>0$ sufficiently small, we get $0 < \kappa:=
(1+\delta)(1-\theta(1-q^{1/(d-1)})) < 1$, and Estimate~\eqref{eq:rho_est2}
becomes
\begin{align*}
  \rho_{\ell+1}^2 &\leq \kappa \rho_\ell^2 + \c{reduction}
  \nnorm{\UU_{\ell+1}-\UU_\ell}^2,
\end{align*}
with $\c{reduction} = (1+\delta^{-1}) C$.

\textbf{Step~6. } Recall that quasi-optimality~\eqref{eq:cea} implies that \\
\mbox{$\lim_{\ell\to\infty} \UU_\ell \in \HH$} exists and thus $\nnorm{\UU_{\ell+1}-\UU_\ell} \rightarrow 0$ as
$\ell \rightarrow \infty$. Together with the estimator
reduction~\eqref{eq:estredest}, elementary calculus predicts convergence
$\rho_\ell \rightarrow 0$ as $\ell \rightarrow \infty$,
cf.~\cite[Section~2]{estconv}.
The reliability $\nnorm{\uu-\UU_\ell}\lesssim\rho_\ell$ of Theorem~\ref{thm:bmc_est} then proves $\UU_\ell \rightarrow
\uu$ for $\ell \rightarrow \infty$.
\qed
\end{proof}

\section{Numerical experiments}\label{sec:numerics}
In this section, we present some numerical 2D experiments, where we compare the
three different FEM-BEM coupling methods on uniform and adaptively generated meshes.
In particular, we emphasize the advantages of adaptive mesh-refinement compared with
uniform refinement.

Firstly, we investigate a linear problem with $\A=\Id$ on an L-shaped domain.
In the second and third experiment, we choose a linear operator $\A$ with $\A
\nabla u = (\c{ellA} \tfrac{\partial u}{\partial x}, \tfrac{\partial u}{\partial
y})$ and ellipticity constant $\c{ellA}>0$. We underline numerically that the assumption $\c{ellA} > 1/4$ in Theorem~\ref{thm:bmc_ell} and Theorem~\ref{thm:bmc_ell_jn} for the unique solvability of the Bielak-MacCamy and Johnson-N\'ed\'elec coupling is sufficient, but \emph{not} necessary.
Finally, we deal with a nonlinear problem on a Z-shaped domain. Throughout, we consider lowest-order elements, i.e.\ $\HH_\ell = \XX_\ell\times\YY_\ell$ with $\XX_\ell = \SS^1(\TT_\ell)$ and $\YY_\ell = \PP^0(\EE_\ell^\Gamma)$.

Let $\zeta_\ell$ be a placeholder for any
of the presented error estimators of Theorem~\ref{thm:bmc_est},~\ref{thm:jn_est}, or~\ref{thm:sym_est}, i.e.\ $\zeta_\ell \in \{ \rho_\ell, \eta_\ell, \mu_\ell \}$. The error estimator $\zeta_\ell$ is split into volume and boundary contributions:
\begin{align}
  \zeta_\ell^2 &= \!\sum_{\tau \in \TT_\ell \cup \EE_\ell^\Omega}\!
  \zeta_\ell(\tau)^2 + \!\sum_{\tau \in \EE_\ell^\Gamma}\! \zeta_\ell(\tau)^2
  =:(\zeta_{\ell}^\Omega)^2 + (\zeta_{\ell}^\Gamma)^2.
\end{align}
Recall that the variable $\phi\in H^{-1/2}(\Gamma)$ has different meanings in the three coupling
methods: In the John\-son-N\'ed\'elec and symmetric coupling $\phi=\nabla u^{\rm
ext}\cdot\NV$ stands for the normal derivative of the exterior solution, whereas
$\phi$ is just a density with $u^{\rm ext} = \widetilde \slp\phi$ in the
Bielak-MacCamy coupling. Quasi-optimality~\eqref{eq:cea} of the coupling methods implies
\begin{align}
\begin{split}
  \nnorm{\uu-\UU_\ell} &\lesssim \norm{u-U_\ell}{H^1(\Omega)} \\
  &\hspace{1cm} + \min\limits_{\Psi_\ell \in \PP^0(\EE_\ell^\Gamma)}
  \norm{\phi-\Psi_\ell}{H^{-1/2}(\Gamma)} \\
  &=: \err_{\ell}^\Omega + \err_{\ell}^\Gamma =: \err_\ell.
\end{split}
\end{align}
Since the variable $\phi$ is not comparable between the different coupling strategies, we only consider the volume terms $\zeta_{\ell}^\Omega$ and $\err_{\ell}^\Omega$ for comparison in the following experiments. Anyhow, we stress that one may expect that the finite element contribution dominates the overall convergence rate.
Optimality of $\err_{\ell}^\Gamma$ and the corresponding estimator contribution $\zeta_{\ell}^\Gamma$ is numerically investigated in~\cite{aposterjn} for the Johnson-N\'ed\'elec coupling and some linear Laplacian in 2D.

In the following, we plot the quantities $\zeta_{\ell}^\Omega$ and
$\err_{\ell}^\Omega$ versus
the number of elements $N = \# \TT_\ell$, where the sequence of meshes
$\TT_\ell$ is obtained with Algorithm~\ref{alg:adaptive}.
We consider adaptive mesh-refinement with adaptivity parameter $\theta = 0.25$
and uniform refinement, which corresponds to the case $\theta=1$.
For all quantities we observe decay rates proportional to $N^{-\alpha}$,
for some $\alpha>0$. We recall that a convergence rate of $\alpha = 1/2$
is optimal for the overall error with P1-FEM.

Moreover, we plot the error quantities $\err_{\ell}^\Omega$ resp.
$\zeta_\ell^\Omega$ versus the computing time $t_\ell$.
The time measurement is different for the uniform and adaptive mesh-refinement:
In the uniform case $t_\ell$, consists of the time which is used to refine the
initial triangulation $\ell$-times, plus the time which is needed to build and solve the
Galerkin system.
In the adaptive case, we set $t_{-1}:=0$. Then, $t_\ell$ consists of the time
$t_{\ell-1}$ needed for all prior steps in the adaptive algorithm, plus the time
needed for one adaptive step on the $\ell$-th mesh, i.e. steps \textit{(i)--(iv)} in
Algorithm~\ref{alg:adaptive}.

All computations were performed on a 64-Bit Linux work station with 32GB of
RAM in \textsc{Matlab} (Release 2009b).
The computation of the boundary integral operators is done with the
\textsc{Matlab} BEM-library \verb+HILBERT+, cf.~\cite{hilbert}.
Throughout, the discrete coupling equations are assembled in the \textsc{Matlab} \texttt{sparse}-format and solved with the \textsc{Matlab} backslash operator.

\begin{figure}[htbp]
\begin{center}
  \includegraphics[width=\columnwidth]{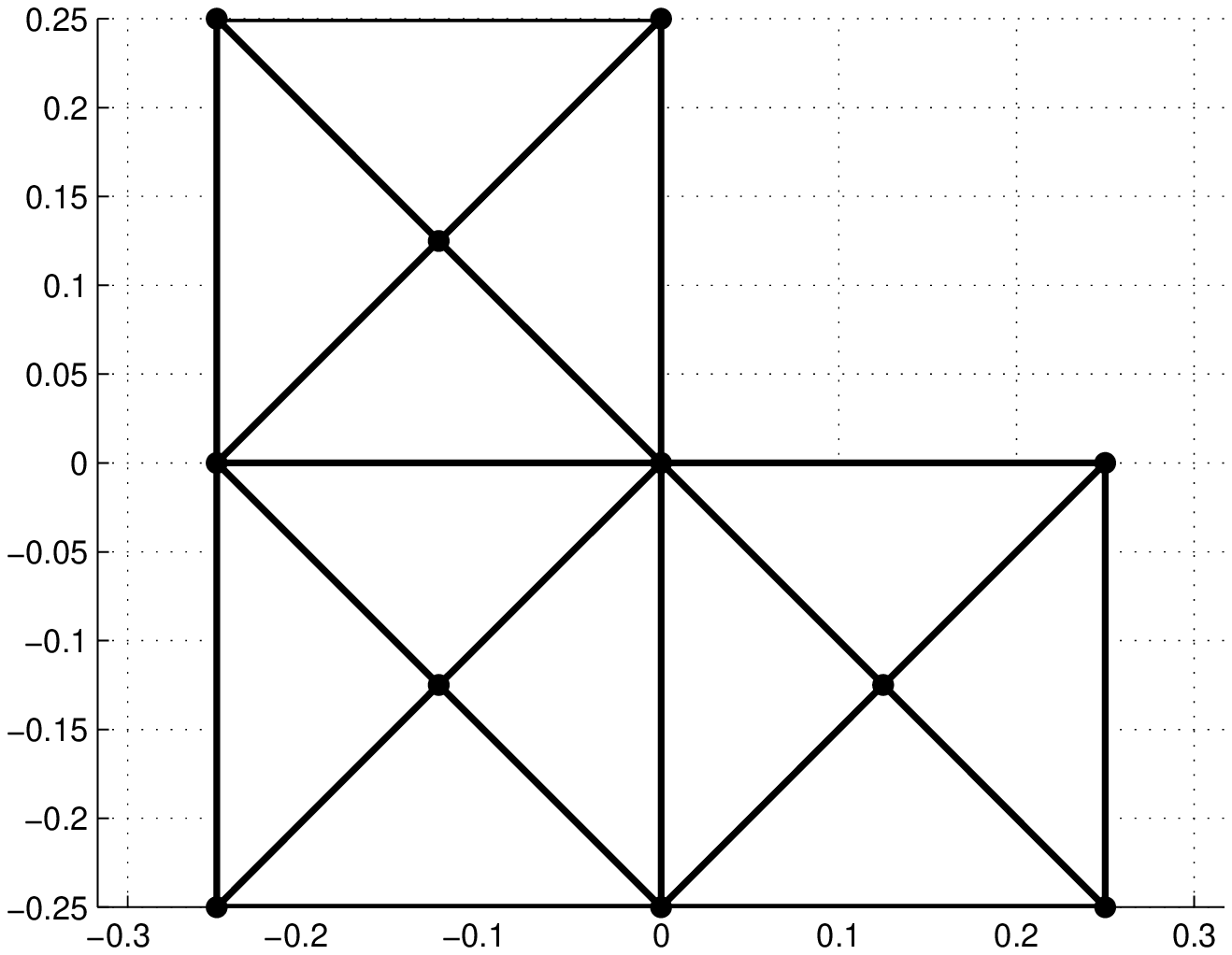}
  \caption{L-shaped domain as well as initial triangulation $\TT_0$ with $\#\TT_0=12$ triangles and initial boundary mesh $\EE_0^\Gamma$ with $\#\EE_0^\Gamma=8$ boundary elements.}
  \label{fig:L_shape}
\end{center}
\end{figure}

\begin{figure}[htbp]
\begin{center}
  \psfrag{error}[c][c]{error contribution $\err_\ell^\Omega$}
  \psfrag{nE}[c][c]{number of elements~$N$}
  \psfrag{convgraphResidual}[c][c]{}
  \psfrag{errbmcadap}{\tiny BMC, adap.}
  \psfrag{errbmcunif}{\tiny BMC, unif.}
  \psfrag{errjnadap}{\tiny JN, adap.}
  \psfrag{errjnunif}{\tiny JN, unif.}
  \psfrag{errsymadap}{\tiny sym, adap.}
  \psfrag{errsymunif}{\tiny sym, unif.}

  \psfrag{OON12}[c][c][1][-25]{\tiny $\OO(N^{-1/2})$}
  \psfrag{OON13}[c][c][1][-17]{\tiny $\OO(N^{-1/3})$}

  \includegraphics[width=\columnwidth]{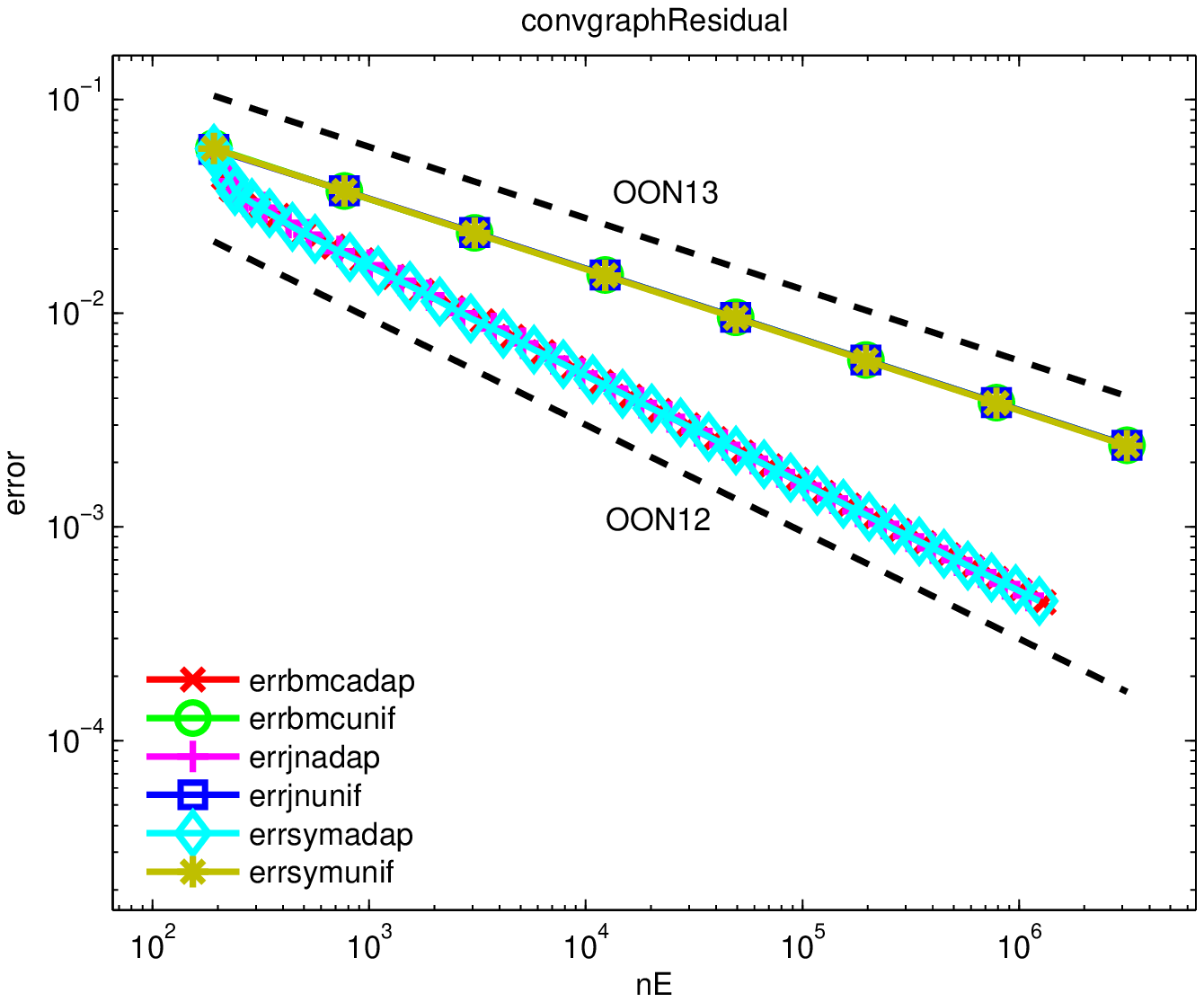}
  \caption{Volume error $\err_\ell^\Omega$ versus number of elements $N$ for Laplace problem of Section~\ref{sec:exp_1}.}
  \label{fig:lin_conv_rates_error}
\end{center}	
\end{figure}

\begin{figure}[htbp]
\begin{center}
  \psfrag{error}[c][c]{estimator contribution $\zeta_\ell^\Omega$}
  \psfrag{nE}[c][c]{number of elements~$N$}
  \psfrag{convgraphResidual}[c][c]{}
  \psfrag{estbmcadap}{\tiny BMC, adap.}
  \psfrag{estbmcunif}{\tiny BMC, unif.}
  \psfrag{estjnadap}{\tiny JN, adap.}
  \psfrag{estjnunif}{\tiny JN, unif.}
  \psfrag{estsymadap}{\tiny sym, adap.}
  \psfrag{estsymunif}{\tiny sym, unif.}

  \psfrag{OON12}[c][c][1][-27]{\tiny $\OO(N^{-1/2})$}
  \psfrag{OON13}[c][c][1][-17]{\tiny $\OO(N^{-1/3})$}

  \includegraphics[width=\columnwidth]{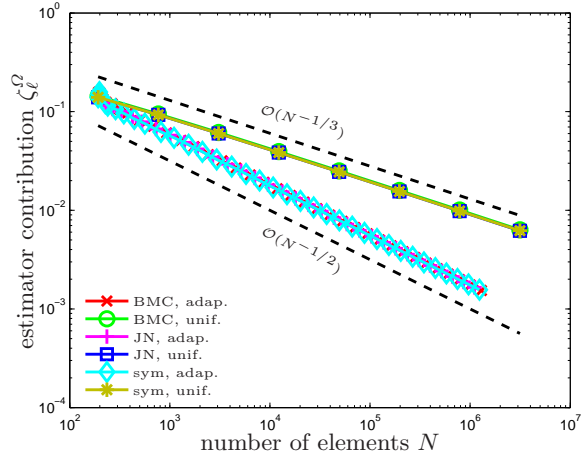}
  \caption{Volume error estimator $\zeta_\ell^\Omega$ versus number of elements
  $N$ for Laplace problem of Section~\ref{sec:exp_1}.}
  \label{fig:lin_conv_rates_est}
\end{center}	
\end{figure}

\begin{figure}[htbp]
\begin{center}
  \psfrag{error}[c][c]{efficiency index $\zeta_\ell^\Omega / \err_\ell^\Omega$}
  \psfrag{nE}[c][c]{number of elements~$N$}
  \psfrag{convgraphResidual}[c][c]{}
  \psfrag{quotbmcadap}{\tiny BMC, adap.}
  \psfrag{quotbmcunif}{\tiny BMC, unif.}
  \psfrag{quotjnadap}{\tiny JN, adap.}
  \psfrag{quotjnunif}{\tiny JN, unif.}
  \psfrag{quotsymadap}{\tiny sym, adap.}
  \psfrag{quotsymunif}{\tiny sym, unif.}

  \includegraphics[width=\columnwidth]{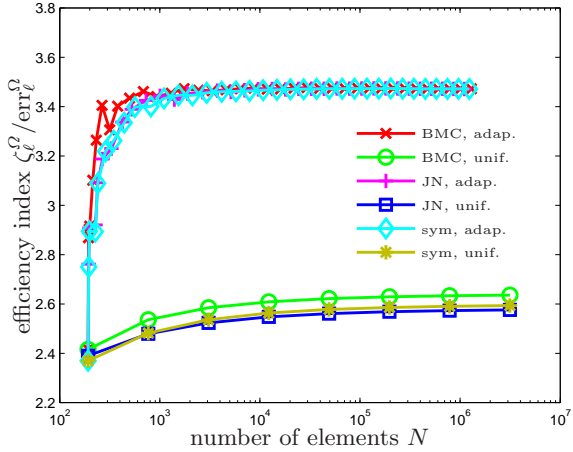}
  \caption{Efficiency index $\zeta_\ell^\Omega / \err_\ell^\Omega$
  versus number of elements $N$ for Laplace problem of Section~\ref{sec:exp_1}. The $x$-axis is scaled logarithmically.}
  \label{fig:lin_quotients}
\end{center}	
\end{figure}

\begin{figure}[htbp]
\begin{center}
  \psfrag{error}[c][c]{error contribution $\err_\ell^\Omega$}
  \psfrag{tinsec}[c][c]{time $t$ [sec.]}
  \psfrag{compTime}[c][c]{}
  \psfrag{errbmcadap}{\tiny BMC, adap.}
  \psfrag{errbmcunif}{\tiny BMC, unif.}
  \psfrag{errjnadap}{\tiny JN, adap.}
  \psfrag{errjnunif}{\tiny JN, unif. }
  \psfrag{errsymadap}{\tiny sym, adap.}
  \psfrag{errsymunif}{\tiny sym, unif. }

  \includegraphics[width=\columnwidth]{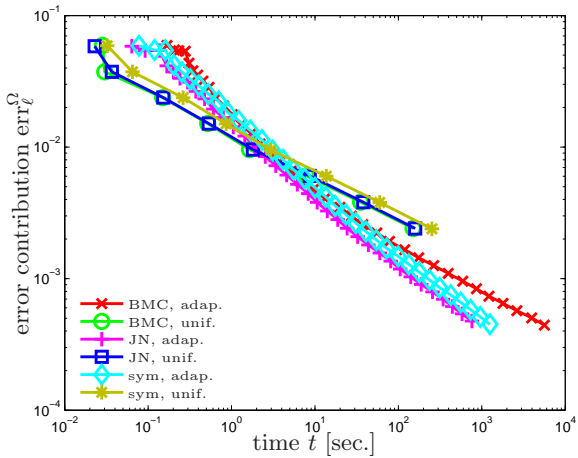}
  \caption{Comparison of the error versus the computing time for
  adaptive and uniform refinement strategy for Laplace problem of Section~\ref{sec:exp_1}.}
  \label{fig:lin_comptime}
\end{center}	
\end{figure}

\subsection{Laplace transmission problem on L-shaped domain}\label{sec:exp_1}

In this experiment, we consider the linear operator $\A = \Id$.
We prescribe the exact solution of~\eqref{eq:strongform} as
\begin{align}\label{eq:exact_u}
  u(x,y) &= r^{2/3} \sin(\tfrac23\varphi), \\
  \label{eq:exact_uext}
  u^{\rm ext}(x,y) &= \tfrac12 \log(|x+\tfrac18|^2 + |y+\tfrac18|^2)
\end{align}
on an L-shaped domain, visualized in Figure~\ref{fig:L_shape}.
Here, $(r,\varphi)$ denote polar coordinates. These functions are then used to
determine the data $(f, u_0, \phi_0)$. Note that $\Delta u = 0 = \Delta u^{\rm ext}$. We stress that $u$ has a generic singularity at the reentrant
corner. Therefore, uniform mesh-refinement leads to a suboptimal convergence
order $\alpha = 1/3$. However, adaptive mesh-refinement recovers the optimal
convergence rate $\alpha = 1/2$. In Figure~\ref{fig:lin_conv_rates_error}
resp.~\ref{fig:lin_conv_rates_est}, we
observe optimal rates for all error and error estimator quantities of the
different coupling methods corresponding to the adaptive scheme,
whereas all quantities corresponding to uniform mesh-refinement converge with
order $\alpha = 1/3$.

The quotients $\zeta_\ell^\Omega / \err_\ell^\Omega$, plotted in Figure~\ref{fig:lin_quotients},
indicate that the residual-based estimators of
Theorem~\ref{thm:bmc_est},~\ref{thm:jn_est} and~\ref{thm:sym_est}
are not only reliable but also efficient: We see that $\zeta_\ell^\Omega /
\err_\ell^\Omega$ is constant for sufficiently large $N$.

When we compare the computing time, see Figure~\ref{fig:lin_comptime}, we
observe that the adaptive strategy is superior to the uniform one.
Moreover, we see differences between the adaptive versions of the
three coupling methods. The Bielak-MacCamy coupling needs
significantly more computing time than the other coupling schemes.
We also observe that the Johnson-N\'ed\'elec coupling is the fastest
of all coupling methods, at least in this experiment.

\subsection{Linear transmission problem on L-shaped domain}\label{sec:exp_2}
We consider a linear problem with $\A = (\c{ellA} \tfrac{\partial u}{\partial
x}, \tfrac{\partial u}{\partial y})$ on an L-shaped domain.
We again prescribe the solutions $(u,u^{\rm ext})$
by~\eqref{eq:exact_u}--\eqref{eq:exact_uext}. Then, $\div \A\nabla
u = (\c{ellA}-1)\tfrac{\partial^2 u}{\partial x^2}$.
In Figure~\ref{fig:ell_A}, we plot the error quantities $\err_{\ell}^\Omega$ of the
adaptive schemes for different values of $\c{ellA}$.
We observe good performance of both the Bielak-MacCamy and Johnson-N\'ed\'elec
coupling also for $\c{ellA} \in (0,1/4]$, which was excluded by our analysis.

\begin{figure}[htbp]
\begin{center}
  \psfrag{error}[c][c]{error contribution $\err_\ell^\Omega$, adaptive}
  \psfrag{nE}[c][c]{number of elements~$N$}
  \psfrag{ellipticity}[c][c]{}

  \psfrag{errbmc001}{\tiny BMC, $\c{ellA} = 0.001$}
  \psfrag{errbmc25}{\tiny BMC, $\c{ellA} = 0.25$}
  \psfrag{errjn001}{\tiny JN, $\c{ellA} = 0.001$}
  \psfrag{errjn25}{\tiny JN, $\c{ellA} = 0.25$}
  \psfrag{errsym001}{\tiny sym, $\c{ellA} = 0.001$}
  \psfrag{errsym25}{\tiny sym, $\c{ellA} = 0.25$}

  \psfrag{OON12}[l][B][1][-20]{\tiny $\OO(N^{-1/2})$}

  \includegraphics[width=\columnwidth]{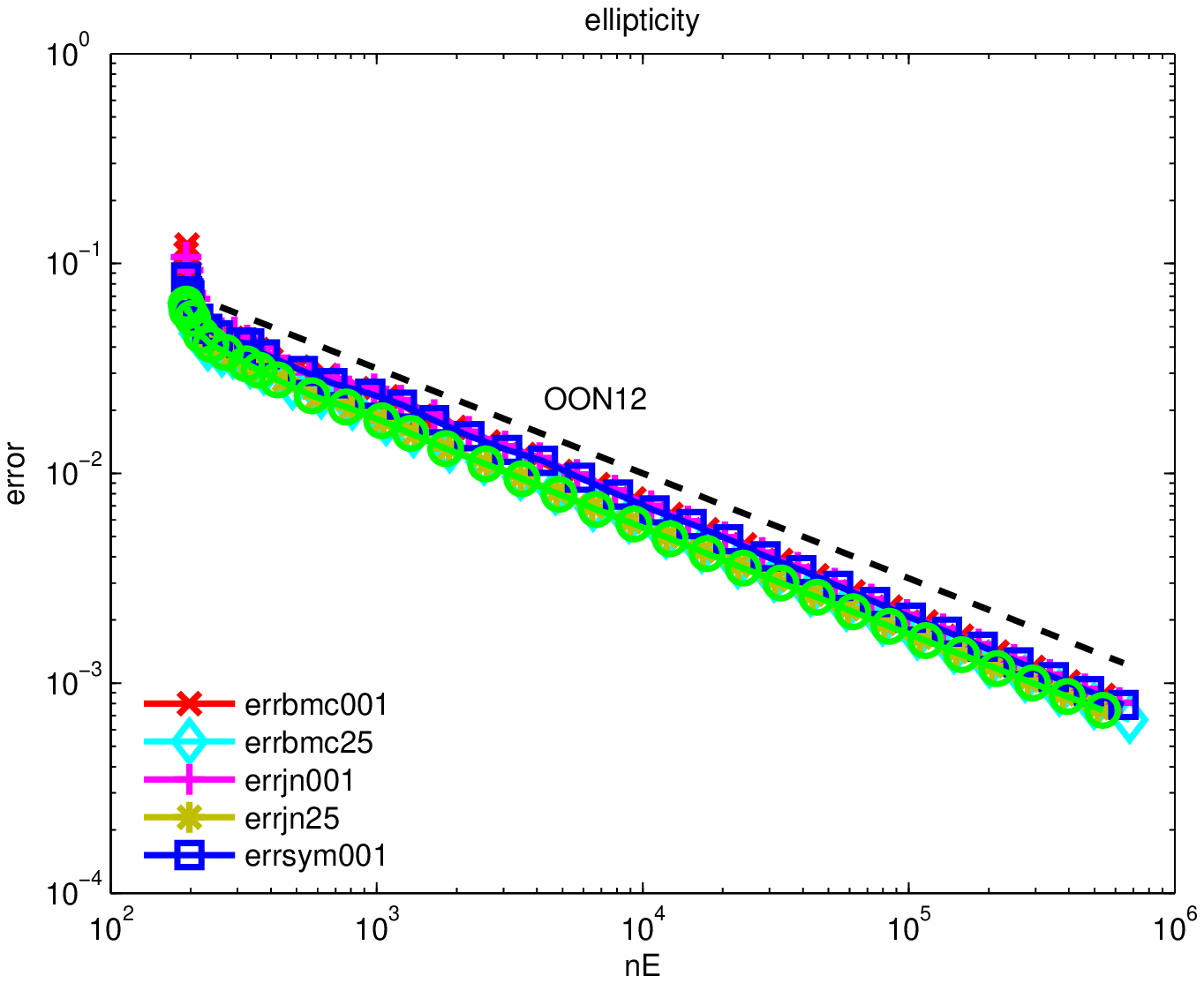}
  \caption{Convergence of the volume errors for different
  values of the ellipticity constant of the operator $\A$ of
  Section~\ref{sec:exp_2} and adaptive mesh-refinement.}
  \label{fig:ell_A}
\end{center}	
\end{figure}

\begin{figure}[htbp]
\begin{center}
  \psfrag{error}[c][c]{error contribution $\err_\ell^\Omega$, uniform}
  \psfrag{nE}[c][c]{number of elements~$N$}
  \psfrag{ellipticity}[c][c]{}

  \psfrag{errbmc001}{\tiny BMC, $\c{ellA} = 0.001$}
  \psfrag{errbmc25}{\tiny BMC, $\c{ellA} = 0.25$}
  \psfrag{errjn001}{\tiny JN, $\c{ellA} = 0.001$}
  \psfrag{errjn25}{\tiny JN, $\c{ellA} = 0.25$}
  \psfrag{errsym001}{\tiny sym, $\c{ellA} = 0.001$}
  \psfrag{errsym25}{\tiny sym, $\c{ellA} = 0.25$}

  \psfrag{OON13}[l][B][1][-19]{\tiny $\OO(N^{-1/3})$}

  \includegraphics[width=\columnwidth]{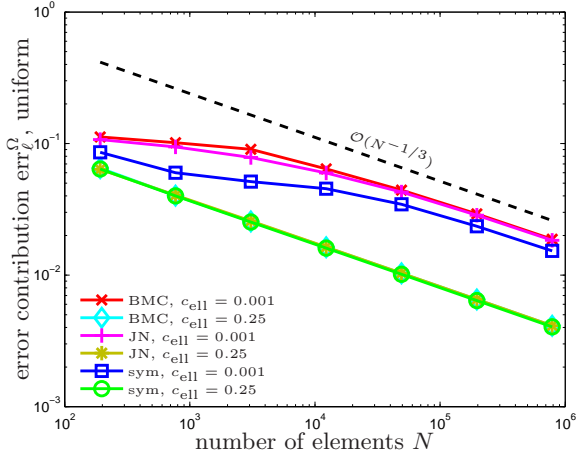}
  \caption{Convergence of the volume errors for different
  values of the ellipticity constant of the operator $\A$ of
  Section~\ref{sec:exp_2} and uniform mesh-refinement.}
  \label{fig:ell_A_unif}
\end{center}	
\end{figure}

\subsection{Linear transmission problem with unknown solution}\label{sec:exp_3}
\begin{figure}[htbp]
\begin{center}
  \includegraphics[width=\columnwidth]{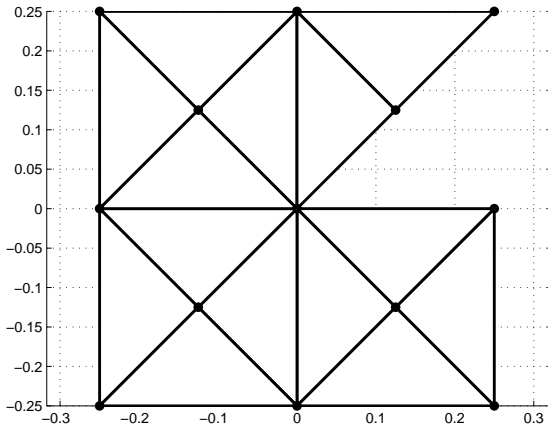}
  \caption{Z-shaped domain as well as initial triangulation $\TT_0$ with $\#
  \TT_0 = 14$ triangles and initial boundary mesh $\EE_0^\Gamma$ with $\#
  \EE_0^\Gamma = 10$ boundary elements.}
  \label{fig:Z_shape}
\end{center}
\end{figure}
We present another experiment to underline the results of the aforegoing
subsection. Again, we consider the operator $\A$ with $\A\nabla u = (\c{ellA}
\tfrac{\partial u}{\partial x}, \tfrac{\partial u}{\partial y})$, but now on a
Z-shaped domain, visualized in Figure~\ref{fig:Z_shape}.
The data is set to $(f, u_0, \phi_0) = (1, 0, 0)$, and we stress that the exact
solution is not known. Therefore, we plot only the error estimator quantities
$\zeta_\ell$ in Figure~\ref{fig:ell_A_2} resp.~\ref{fig:ell_A_2_unif}.
An optimal convergence order $\alpha = 1/2$ for the estimators corresponding to
the adaptive schemes
is observed, whereas uniform refinement methods lead to suboptimal convergence rates.
As in Section~\ref{sec:exp_2}, Figure~\ref{fig:ell_A_2}
resp.~\ref{fig:ell_A_2_unif} indicates a good performance of both the
Johnson-N\'ed\'elec and Bielak-MacCamy coupling for $\c{ellA} \in (0,1/4]$.

\begin{figure}[htbp]
\begin{center}
  \psfrag{error}[c][c]{error estimator $\zeta_\ell$, adaptive}
  \psfrag{nE}[c][c]{number of elements~$N$}
  \psfrag{ellipticity}[c][c]{}

  \psfrag{errbmc001}{\tiny BMC, $\c{ellA} = 0.001$}
  \psfrag{errbmc25}{\tiny BMC, $\c{ellA} = 0.25$}
  \psfrag{errjn001}{\tiny JN, $\c{ellA} = 0.001$}
  \psfrag{errjn25}{\tiny JN, $\c{ellA} = 0.25$}
  \psfrag{errsym001}{\tiny sym, $\c{ellA} = 0.001$}
  \psfrag{errsym25}{\tiny sym, $\c{ellA} = 0.25$}

  \psfrag{OON12}[l][B][1][-27]{\tiny $\OO(N^{-1/2})$}

  \includegraphics[width=\columnwidth]{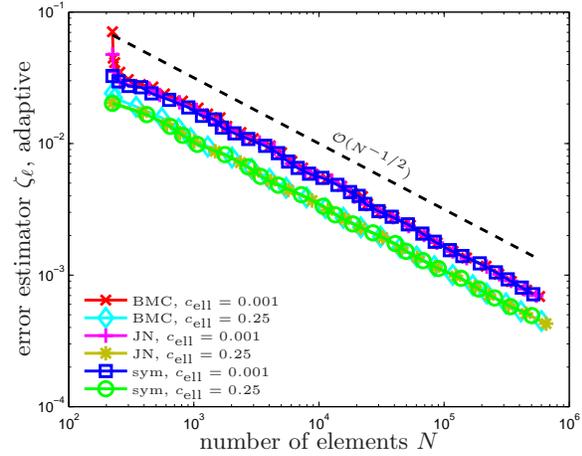}
  \caption{Convergence of the respective error estimators $\zeta_\ell$ for different
  values of the ellipticity constant of the operator $\A$ of
  Section~\ref{sec:exp_3} and adaptive mesh-refinement.}
  \label{fig:ell_A_2}
\end{center}	
\end{figure}

\begin{figure}[htbp]
\begin{center}
  \psfrag{error}[c][c]{error estimator $\zeta_\ell$, uniform}
  \psfrag{nE}[c][c]{number of elements~$N$}
  \psfrag{ellipticity}[c][c]{}

  \psfrag{errbmc001}{\tiny BMC, $\c{ellA} = 0.001$}
  \psfrag{errbmc25}{\tiny BMC, $\c{ellA} = 0.25$}
  \psfrag{errjn001}{\tiny JN, $\c{ellA} = 0.001$}
  \psfrag{errjn25}{\tiny JN, $\c{ellA} = 0.25$}
  \psfrag{errsym001}{\tiny sym, $\c{ellA} = 0.001$}
  \psfrag{errsym25}{\tiny sym, $\c{ellA} = 0.25$}

  \psfrag{OON12}[l][B][1][-22]{\tiny $\OO(N^{-1/2})$}
  \psfrag{OON13}[l][B][1][-15]{\tiny $\OO(N^{-1/3})$}

  \includegraphics[width=\columnwidth]{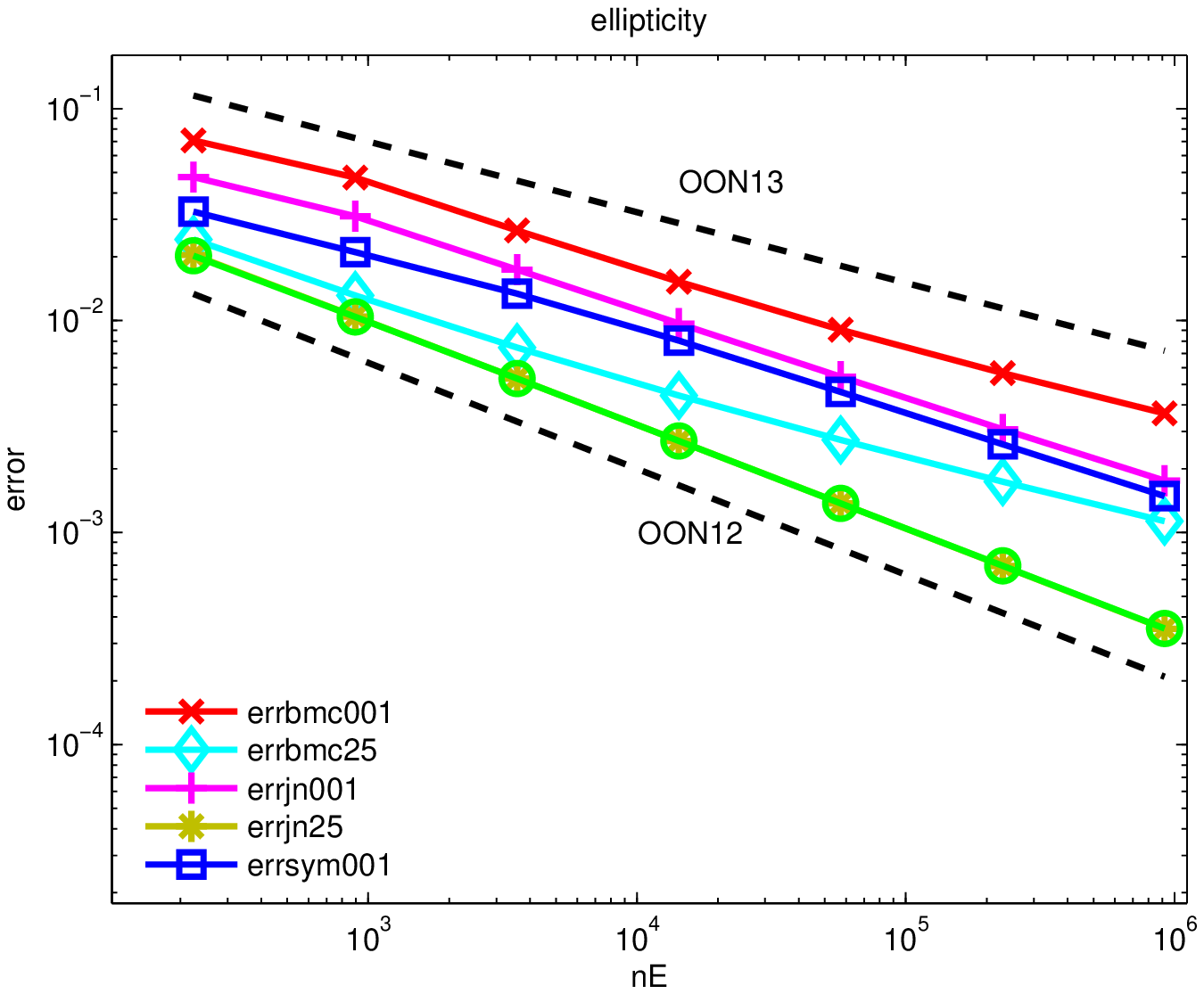}
  \caption{Convergence of the respective error estimators $\zeta_\ell$ for different
  values of the ellipticity constant of the operator $\A$ of
  Section~\ref{sec:exp_3} and uniform mesh-refinement.}
  \label{fig:ell_A_2_unif}
\end{center}	
\end{figure}

\subsection{Nonlinear experiment on Z-shaped domain}\label{sec:exp_nonlin}
In the last experiment we consider a Z-shaped domain, visualized in
Figure~\ref{fig:Z_shape}, and a nonlinear operator $\A$ with $\A\nabla u =
g(|\nabla u|) \nabla u$, where $g(t) = 2+1/(1+t)$ for $t\ge0$.
Note that the ellipticity constant of $\A$ is $\c{ellA} = 2$.
The prescribed solution

\begin{align}
  u(x,y) &= r^{4/7} \sin(\tfrac47\varphi), \\
  u^{\rm ext}(x,y) &= \tfrac{x+y+0.25}{(x+0.125)^2 + (y+0.125)^2}
\end{align}
of~\eqref{eq:strongform} fulfills $\Delta u^{\rm ext} = 0$. The interior solution
$u$ has a generic singularity at the reentrant corner.
As can be seen in Figure~\ref{fig:nonlin_conv_rates_error},
resp.~\ref{fig:nonlin_conv_rates_est}, the uniform strategy
leads to a suboptimal convergence rate $\alpha = 2/7$, whereas the adaptive
strategy leads to the optimal convergence rate $\alpha = 1/2$.

\begin{figure}[htbp]
\begin{center}
  \psfrag{error}[c][c]{error contribution $\err_\ell^\Omega$}
  \psfrag{nE}[c][c]{number of elements $N$}
  \psfrag{convgraphResidual}[c][c]{}
  \psfrag{errbmcadap}{\tiny BMC, adap.}
  \psfrag{errbmcunif}{\tiny BMC, unif.}
  \psfrag{errjnadap}{\tiny JN, adap.}
  \psfrag{errjnunif}{\tiny JN, unif.}
  \psfrag{errsymadap}{\tiny sym, adap.}
  \psfrag{errsymunif}{\tiny sym, unif.}

  \psfrag{OON12}[c][c][1][-27]{\tiny $\OO(N^{-1/2})$}
  \psfrag{OON27}[c][c][1][-15]{\tiny $\OO(N^{-2/7})$}

  \includegraphics[width=\columnwidth]{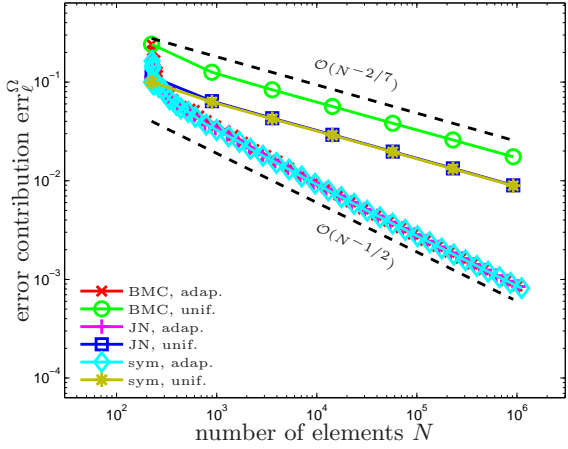}
  \caption{Convergence of the volume error $\err_\ell^\Omega$ for the nonlinear
  experiment of Section~\ref{sec:exp_nonlin}.}
  \label{fig:nonlin_conv_rates_error}
\end{center}	
\end{figure}

\begin{figure}[htbp]
\begin{center}
  \psfrag{error}[c][c]{estimator contribution $\zeta_\ell^\Omega$}
  \psfrag{nE}[c][c]{number of elements $N$}
  \psfrag{convgraphResidual}[c][c]{}
  \psfrag{estbmcadap}{\tiny BMC, adap.}
  \psfrag{estbmcunif}{\tiny BMC, unif.}
  \psfrag{estjnadap}{\tiny JN, adap.}
  \psfrag{estjnunif}{\tiny JN, unif.}
  \psfrag{estsymadap}{\tiny sym, adap.}
  \psfrag{estsymunif}{\tiny sym, unif.}

  \psfrag{OON12}[c][c][1][-27]{\tiny $\OO(N^{-1/2})$}
  \psfrag{OON27}[c][c][1][-15]{\tiny $\OO(N^{-2/7})$}

  \includegraphics[width=\columnwidth]{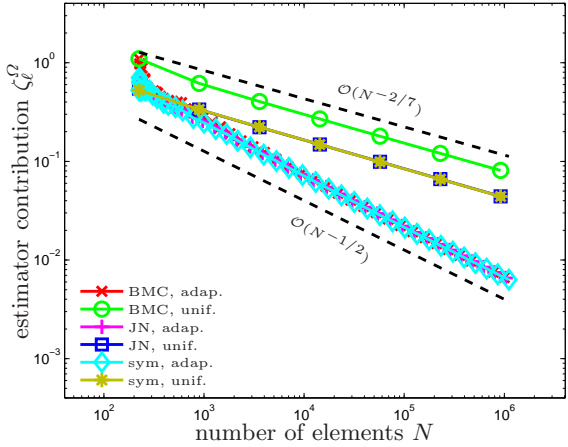}
  \caption{Convergence of the volume error estimator quantities $\zeta_\ell^\Omega$ for the
  nonlinear
  experiment of Section~\ref{sec:exp_nonlin}.}
  \label{fig:nonlin_conv_rates_est}
\end{center}	
\end{figure}

\begin{figure}[htbp]
\begin{center}
  \psfrag{error}[c][c]{efficiency index $\zeta_\ell^\Omega / \err_\ell^\Omega$}
  \psfrag{nE}[c][c]{number of elements $N$}
  \psfrag{convgraphResidual}[c][c]{}
  \psfrag{quotbmcadap}{\tiny BMC, adap.}
  \psfrag{quotbmcunif}{\tiny BMC, unif.}
  \psfrag{quotjnadap}{\tiny JN, adap.}
  \psfrag{quotjnunif}{\tiny JN, unif.}
  \psfrag{quotsymadap}{\tiny sym, adap.}
  \psfrag{quotsymunif}{\tiny sym, unif.}

  \includegraphics[width=\columnwidth]{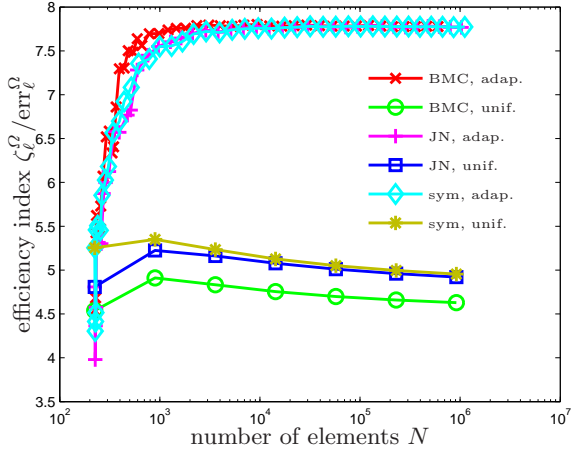}
  \caption{Efficiency index $\zeta_\ell^\Omega / \err_\ell^\Omega$
  versus number of elements $N$ for the nonlinear experiment of Section~\ref{sec:exp_nonlin}.
  The $x$-axis is scaled logarithmically.}
  \label{fig:nonlin_quotients}
\end{center}	
\end{figure}

\begin{figure}[htbp]
\begin{center}
  \psfrag{error}[c][c]{error}
  \psfrag{tinsec}[c][c]{time $t$ [sec.]}
  \psfrag{compTime}[c][c]{}
  \psfrag{error}[c][c]{error contribution $\err_\ell^\Omega$}

  \psfrag{compTime}[c][c]{}
  \psfrag{errbmcadap}{\tiny BMC, adap.}
  \psfrag{errbmcunif}{\tiny BMC, unif.}
  \psfrag{errjnadap}{\tiny JN, adap.}
  \psfrag{errjnunif}{\tiny JN, unif. }
  \psfrag{errsymadap}{\tiny sym, adap.}
  \psfrag{errsymunif}{\tiny sym, unif. }

  \includegraphics[width=\columnwidth]{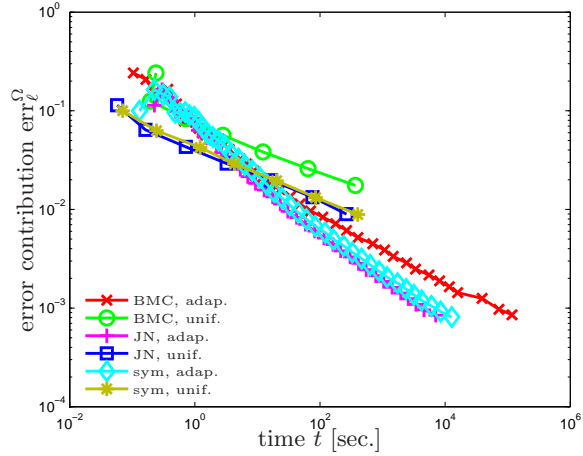}
  \caption{Comparison of the volume error versus the computing time for the
  nonlinear experiment of Section~\ref{sec:exp_nonlin} and uniform resp.\ adaptive mesh-refinement.}
  \label{fig:nonlin_comptime}
\end{center}	
\end{figure}

The results of Figure~\ref{fig:nonlin_quotients} argue for the efficiency of the
reliable error estimator $\zeta_\ell \in \{\rho_\ell, \eta_\ell, \zeta_\ell \}$,
which matches our observations of the first experiment.
As in Section~\ref{sec:exp_1}, we obtain from Figure~\ref{fig:nonlin_comptime}
that the adaptive strategy is superior to the uniform one.

\subsection{Conclusion}
In all our experiments, we observe that the adaptive mesh-refinement strategy empirically
leads to optimal convergence rates for the error as well as for the error
estimator quantities, whereas uniform refinement is inferior to the adaptive
ones.
Moreover, we see that the error quantities for the three different coupling
methods differ only slightly.
From this point of view, one cannot favour a certain coupling method.
But we stress that there are significant differences in the computing
times, where the Johnson-N\'ed\'elec coupling is superior to the other two
coupling methods. Although the Bielak-MacCamy coupling is just the transposed
problem of the Johnson-N\'ed\'elec coupling in the linear case, it
needs the most computing time of all three coupling schemes.
Furthermore, we have numerical evidence that the reliable error estimators
corresponding to the different coupling schemes are also efficient.

Finally, we remark that our numerical experiments have shown that the assumption
$\c{ellA} > 1/4$ from Theorem~\ref{thm:bmc_ell} and~\ref{thm:bmc_ell_jn} is sufficient for the
solvability of the Bielak-MacCamy coupling and Johnson-N\'ed\'elec coupling, but
\textit{not} necessary.


\begin{acknowledgements}
  The research of the authors Markus Aurada, Michael Feischl, Michael Karkulik
  and Dirk Praetorius is supported
  through the FWF project \textit{Adaptive Boundary Element Method}, funded by
  the Austrian Science fund (FWF) under grant P21732.
\end{acknowledgements}

\end{document}